\documentclass[12pt,a4paper]{amsart}
\usepackage{amssymb,amscd,enumerate,array}

\title[The geometry of second-order equations]{The geometry of second-order
  ordinary differential equations\footnote{Lecture notes of the Mathematical School on the geometry of differential equations, Nordfjordeid, Norrway, June 1996}}
\author{Boris Doubrov \and Boris Komrakov}

\newtheorem{lem}{Lemma}[section]
\newtheorem{thr}{Theorem}

\newtheorem{prop}{Proposition}
\newtheorem*{cor}{Corollary}

\theoremstyle{definition}
\newtheorem*{defi}{Definition}
\newtheorem*{ex}{Example}
\newtheorem*{exmps}{Examples}

\theoremstyle{remark}
\newtheorem{rem}{Remark}[section]
\newtheorem*{rems}{Remarks}
\newtheorem*{demothr}{Proof of Theorem~2}
\newcommand\R{\mathbb R}
\newcommand\C{\mathbb C}

\newcommand\E{\mathcal E}
\newcommand\M{\mathcal M}
\newcommand\D{\mathcal D}
\newcommand{\Vect}{\mathcal D}
\newcommand\F{\mathcal F}
\renewcommand\L{\Lambda}
\renewcommand\v{\varphi}
\renewcommand\l{\lambda}
\newcommand\al{\alpha}
\renewcommand\a{\alpha}
\renewcommand\b{\beta}
\newcommand\p{\partial}
\newcommand{\G}{\overline G}
\renewcommand{\P}{\overline P}

\newcommand{\g}{\mathfrak g}
\newcommand{\gb}{\bar\g}
\newcommand{\h}{\mathfrak h}

\renewcommand{\aa}{\mathfrak a}
\newcommand{\codim}{\operatorname{codim}}
\newcommand{\om}{\omega}
\newcommand{\Om}{\Omega}
\newcommand{\tom}{\tilde\omega}
\newcommand{\TOm}{\widetilde\Omega}
\newcommand{\TX}{\widetilde X}
\newcommand{\bom}{\bar\om}
\newcommand{\Ad}{\operatorname{Ad}}
\renewcommand{\d}{\,d}
\newcommand{\Mat}{\operatorname{Mat}}
\newcommand{\tr}{\operatorname{tr}}
\newcommand{\const}{\mathrm{const}}
\newcommand{\id}{\operatorname{id}}
\newcommand{\sym}{\operatorname{sym}}
\newcommand{\im}{\operatorname{Im}}
\newcommand{\ov}{\overline}

\newcommand\Lgg{\mathcal L(\gb/\g\land\gb/\g,\gb)}
\newcommand{\po}[1]{\frac{\partial}{\partial #1}}
\newcommand{\arctg}{\arctan}
\begin{document}

\maketitle
\tableofcontents
\section{Introduction to second-order differential equations}

\subsection{The space of first jets in the plane}

By the plane $\R^2$ we mean a smooth manifold with a fixed coordinate system
$(x,y)$.

\begin{defi} The {\it space of first jets\/}, denoted $J^1(\R^2)$, is the set
  of all one-dimensional subspaces (directions) in the tangent spaces
  to~$\R^2$.
\end{defi}

In other words,
$$J^1(\R^2)=\{l_p\mid l_p\text{~is a one-dimensional subspace in~}T_p\R^2\}.$$
Let $\pi\colon J^1(\R^2)\to\R^2$ denote the natural projection which takes
$l_p$ to the corresponding point~$p$ of the plane.

We shall now introduce two local coordinate systems $(x,y,z_1)$ and
$(x,y,z_2)$ in $J^1(\R^2)$. To this end, we note that every direction $l_p$ at
the point $p=(x,y)$ is generated by a nonzero tangent vector $\a\frac\p{\p x}+
\b\frac\p{\p y}$, $\a,\b\in\R$, which is unique up to a constant factor.
Define
\begin{align*}
&U_1=\left\{l_p=\left\langle\a\frac\p{\p x}+\b\frac\p{\p y}\right\rangle
\bigg|\,\a\ne0\right\},\\ 
&U_2=\left\{l_p=\left\langle\a\frac\p{\p x}+\b\frac\p{\p y}\right\rangle
\bigg|\,\b\ne0\right\}. 
\end{align*}
Then $(x,y,z_1=\frac\b\a)$ and $(x,y,z_2=\frac\a\b)$ will be coordinates in 
$U_1$ and~$U_2$ respectively. If we identify $T_p\R^2$ with $\R^2$, then
$U_1$ will contain all directions which are not parallel to the $y$-axis, while
$U_2$ will consist of all those which are not parallel to the $x$-axis. 
The transition function from $U_1$
to~$U_2$ has obviously the form
$$\v_{12}\colon (x,y,z_1)\mapsto(x,y,\frac1{z_1}).$$

These coordinate charts make $J^1(\R^2)$ into a three-dimensional smooth
manifold. In the following, unless otherwise stated, we shall use the local
coordinate system $(x,y,z_1)$ in $J^1(\R^2)$, denoting it simply by $(x,y,z)$.

\begin{defi} Let $N$~ be a one-dimensional submanifold ($\equiv$~%
a non-parametrized curve) in the plane. The \emph{prolongation}
$N^{(1)}$ of~$N$ is a curve in $J^1(\R^2)$ which has the form
$$N^{(1)}=\{T_pN\mid p\in N\}.$$
\end{defi}

\begin{ex}
If $N$~is the graph of some function $y(x)$, that is, if
$N=\big\{\big(x,y(x)\big)\mid x\in\R\big\}$, then, in local coordinates,
$N^{(1)}$ has the form
\begin{equation}\label{eq1}
\big\{\big(x,y(x),y'(x)\big)\mid x\in\R\big\}. 
\end{equation}
\end{ex} 
The tangent space to the curve~(\ref{eq1}) at the point $(x,y,z)=
\big(x,y(x),y'(x)\big)$ is generated by the vector
\begin{equation}\label{eq2}
\frac\p{\p x}+y'(x)\frac\p{\p y}+y''(x)\frac\p{\p z}=
\frac\p{\p x}+z\frac\p{\p y}+y''(x)\frac\p{\p z}.
\end{equation}
Note that, irrespective of the function $y(x)$, this tangent space always lies
in the two-dimensional subspace generated by the vectors 
$\frac\p{\p x}+z\frac\p{\p y}$ and $\frac\p{\p z}$.

\begin{defi} The {\it contact distribution\/}~$C$ on $J^1(\R^2)$ is the family 
$$C_q=\langle T_qN^{(1)}\mid N^{(1)}\ni q\rangle$$
of two-dimensional subspaces in the tangent spaces to $J^1(\R^2)$.
\end{defi}
From~(\ref{eq2}) it immediately follows that in local coordinates
$$C_q=\left\langle\frac\p{\p x}+z\frac\p{\p y},\frac\p{\p z}\right\rangle.$$
Notice that the field $\frac\p{\p z}$ is tangent to the fibers of the
projection $\pi\colon J^1(\R^2)\to\R^2$, so that $\ker d_q\pi\subset C_q$ for
all $q\in J^1(\R^2)$.

%%%%  picture

By the definition of the contact distribution, the prolongation $N^{(1)}$ of
any curve~$N$ in the plane is tangent to the contact distribution.  It turns
out that the converse is also true.

\begin{lem} If $M$~is a curve in $J^1(\R^2)$ tangent to the contact
  distribution and such that the projection $\pi_{|M}\colon M\to\R^2$ is
  nondegenerate, then, viewed locally, $M$ is the prolongation of some curve
  in the plane.
\end{lem}
\begin{proof} Without loss of generality we may assume that the projection
of~$M$ onto $\R^2$, locally, is the graph of some function $y(x)$. Then
$$M=\big\{\big(x,y(x),z(x)\big)\mid x\in\R\big\}.$$
But the vector $\frac\p{\p x}+y'\frac\p{\p x}+z'(x)\frac\p{\p x}$ lies in 
$C_{(x,y(x),z(x))}$ if and only if $z=y'$ and $M$ is the prolongation of the
curve $N=\big\{\big(x,y(x)\big)\mid x\in\R\big\}$.
\end{proof}

Now let $\v\colon\R^2\to\R^2$~be an arbitrary diffeomorphism of the plane.

\begin{defi} The diffeomorphism $\v^{(1)}\colon J^1(\R^2)\to J^1(\R^2)$ defined
  by 
  $$\v^{(1)}\colon l_p\mapsto (d_p\v)(l_p)$$ 
  is said to be the {\it prolongation\/} of~$\v$.
\end{defi}

It is easily verified that the following diagram is commutative:
$$\begin{CD}
J^1(\R^2) @>\v^{(1)}>> J^1(\R^2)\\
@VV\pi V @VV\pi V\\
\R^2 @>\v>>\R^2
\end{CD}$$

\begin{lem} For any diffeomorphism $\v$ the prolongation  $\v^{(1)}$ 
  preserves the contact distribution $C$ (i.e., 
  $d_q\v^{(1)}(C_q)=C_{\v^{(1)}(q)}$ for all $q\in J^1(\R^2)$).
\end{lem}
\begin{proof} The proof is immediate from the definition of the contact
  distribution and the fact that $\v^{(1)}(N^{(1)})=\v(N)^{(1)}$ for any curve
  $N\subset\R^2$.
\end{proof}

To find the expression for $\v^{(1)}$ in local coordinates, we assume that 
$\v(x,y)=(A(x,y),B(x,y))$ for some smooth functions $A$ and $B$ in the plane.
Then $\v^{(1)}$ has the form
$$ \v^{(1)}\colon (x,y,z)\mapsto (A(x,y),B(x,y),C(x,y,z)).$$
If we explicitly write down the condition that the contact distribution is
invariant under $\v^{(1)}$, we get
$$C(x,y,z)=\frac{B_x+B_yz}{A_x+A_yz},$$
where $A_x,A_y,B_x,B_y$ denote the partial derivatives of $A$ and $B$ with
respect to $x$ and $y$.

\subsection{Second-order equations in the plane}\label{secord:pl}

Let
\begin{equation}\label{eq3}
y''=F(x,y,y')
\end{equation}
be an arbitrary second-order differential equation solved for the highest
derivative. Consider the direction field~$E$ in the chart $U_1$ with
coordinates $(x,y,z)$ defined by
$$E_{(x,y,z)}=\left\langle\frac\p{\p x}+z\frac{\p}{\p y}+F(x,y,z)\frac
\p{\p z}\right\rangle.$$

\begin{lem} If $N$~is the graph of the function~$y(x)$, then $y(x)$~is a
  solution of equation~{\rm(\ref{eq3})} if and only if $N^{(1)}$~is an integral
  curve of the direction field~$E$.
\end{lem}
\begin{proof} Suppose $y(x)$~is a solution of equation~(\ref{eq3}). Then
$N^{(1)}=\big\{\big(x,y(x),y'(x)\big)\mid x\in\R\big\}$, and the tangent space
to~$N^{(1)}$ is given by the vector
$$\frac\p{\p x}+y'\frac{\p}{\p x}+y''(x)\frac\p{\p x}=
\frac\p{\p x}+y'\frac{\p}{\p y}+F(x,y,y')\frac\p{\p z}.$$
Now since the coordinate $z$ is equal to $y'$ for all points on~$N^{(1)}$, we
see that $N^{(1)}$~is an integral manifold for~$E$.

The converse is proved in a similar manner.
\end{proof}

Note that $E$ is contained in the contact distribution and that the projection
of~$E$ onto $\R^2$ by means of $d\pi$ is always nondegenerate. Any integral
curve of the direction field~$E$ is therefore a lift of the graph of a solution
of equation~(\ref{eq3}).

Thus, {\bf there is a one-to-one correspondence between the solutions of
  equation~(\ref{eq3}) and the integral curves of the direction field~$E$: the
  graphs of the solutions are projections of the integral curves, while the
  integral curves are  of the graphs of the solutions.}

Let us now consider an arbitrary direction field~$E$ in~$J^1(\R^2)$ such that
\begin{enumerate}
\item[(i)] $E$~is contained in the contact distribution;
\item[(ii)] at each point $q\in J^1(\R^2)$, $E_q$ does not coincide with the
  \emph{vertical} field $V_q=\ker d_q\pi$.
\end{enumerate}
Then $E$ may be regarded as a second-order equation whose solutions are curves
in the plane: a curve $N\subset\R^2$ is by definition a solution of this
equation if $N^{(1)}$ is an integral curve of the field~$E$.
Note that now solutions need no longer be graphs of functions~$y(x)$.

Thus, geometrically, the second-order differential equations (solved for the
highest derivative) may be identified with the direction fields in $J^1(\R^2)$
satisfying  conditions (i) and~(ii). 

Consider how (local) diffeomorphisms of the plane (i.e., ``changes of
variables'' $x,y$) act on these direction fields. Let $E$ be a direction field
contained in the contact distribution and corresponding to some second-order
differential equation in the plane, and let $\v$~be a local diffeomorphism of
the plane. Then $\v$~extends to some local transformation $\v^{(1)}$ of
$J^1(\R^2)$, which acts on the field~$E$ in the following way
$$\big(\v^{(1)}.E\big)_{\v^{(1)}(q)}=d_q\v^{(1)}(E_q)$$
for all $q\in J^1(\R^2)$. It is easy to show that the resulting direction field
still satisfies conditions (i) and~(ii) and defines a new second-order
differential equation.

Let $E_1,~E_2$ be two directions fields contained in the contact distribution
and corresponding to two second-order differential equations. We shall say that
these equations are {\em (locally) equivalent\/} if there exists a (local)
diffeomorphism~$\v$ of the plane such that its prolongation $\v^{(1)}$ takes 
$E_1$ into~$E_2$.

\subsection{Pairs of direction fields in space}

Let $V$~be the vertical direction field in $J^1(\R^2)$ defined by $V_q=\ker
d_q\pi$ for all $q\in J^1(\R^2)$. As we mentioned before, the direction
field~$V$ is contained in the tangent distribution~$C$. In accordance with the
previous subsection, a second-order differential equation can be considered as
another direction field~$E$, contained in~$C$, such that
$$C_q=V_q\oplus E_q\quad\text{for all }q\in J^1(\R^2).$$ 
Conversely, suppose we are given two arbitrary direction fields $E_1$ and~$E_2$
in space that differ at each point of the space. Then $E_1$ and~$E_2$ define
a two-dimensional distribution, denoted $\widetilde C=E_1\oplus E_2$. We say
that a pair of direction fields $E_1$ and $E_2$ is {\it nondegenerate} if the
corresponding distribution~$\widetilde C$ is not completely integrable.

We shall now show that the problem of local classification of pairs of
direction fields in space is equivalent to local classification of second-order
differential equations up to diffeomorphisms of the plane. We shall need the
following theorem.
\begin{thr}\label{thr1} 
If $\widetilde C$~is a two-dimensional, not completely integrable distribution
in space and $\widetilde V$~is an arbitrary direction field contained
in~$\widetilde C$, then the pair $(\widetilde C,\widetilde V)$ 
is locally equivalent to the pair $(C,V)$ in the space $J^1(\R^2)$, that is, 
there exists a local diffeomorphism $\v\colon\R^3\to J^1(\R^2)$ such that
$d\v(\widetilde C)=C$ and $d\v(\widetilde V)=V$.
\end{thr}
\begin{proof} See, for example, \cite[Chapter~11]{olver2}, \cite{gar}.
\end{proof}

Now let $E_1,E_2$~be a nondegenerate pair of direction fields in space and let
$\widetilde C=E_1\oplus E_2$. If $\v\colon\R^3\to J^1(\R^2)$ is a local
diffeomorphism establishing the equivalence of the pairs
$(\widetilde C,E_1)$ and $(C,V)$, then the direction field $E=d\v(E_2)$
determines a second-order equation in the plane.

\begin{thr}\label{thr2} Two nondegenerate pairs of direction fields are locally
  equivalent if and only if so are the corresponding second-order differential
  equations in the plane.
\end{thr}

Before proceeding to the proof of Theorem~\ref{thr2}, we establish one
auxiliary result.

\begin{defi} A transformation of the space $J^1(\R^2)$ is called \emph{contact}
  if it preserves the contact distribution.
\end{defi}

For example, the prolongation $\v^{(1)}$ of any diffeomorphism
$\v\colon\R^2\to\R^2$ is a contact transformation.

\begin{lem} A contact transformation has the form~$\v^{(1)}$, where $\v$ is a
  diffeomorphism of the plane, if and only if it preserves the vertical
  direction field~$V$.
\end{lem}
\begin{proof} The necessity of this condition is obvious. Assume that $\psi$ is
  a contact transformation preserving~$V$. Then $\psi$ defines, in the obvious
  way, a diffeomorphism~$\v$ of the plane. Straightforward computation shows
  that $\psi$ is uniquely determined by~$\v$ and hence coincides
  with~$\v^{(1)}$.
\end{proof}

\begin{demothr}
Let $(E_1,E_2)$ and $(E_1',E_2')$~be two nondegenerate pairs of direction
fields in space, and let $E$ and~$E'$ be the direction fields, contained
in~$C$, that determine the corresponding differential equations. It is clear
that these two pairs of direction fields are locally equivalent if and only if
so are the pairs $(V,E)$ and $(V,E')$.

If $\psi$~is a local diffeomorphism of the space $J^1(\R^2)$ carrying the pair 
$(V,E)$ into the pair $(V,E')$, then $\psi$ preserves the contact distribution 
$C=V\oplus E=V\oplus E'$ and, at the same time, leaves invariant the direction
field~$V$. Hence there exists a local diffeomorphism~$\v$ of the plane such
that $\psi=\v^{(1)}$. Then, by definition, $\v$ establishes the equivalence of
the corresponding second-order differential equations.

The converse is obvious. \hfill\qed
\end{demothr}

\subsection{Duality}

Suppose that a direction field $E$, contained in the contact distribution of
$J^1(\R^2)$, determines a second-order equation in the plane. Then, by
Theorem~1, there exists a local diffeomorphism~$\v$ of the space $J^1(\R^2)$
that takes the pair of distributions $(C,E)$ into the pair $(C,V)$, so that the
direction field~$E$ becomes vertical, while the vertical direction field~$V$
is transformed in some direction field~$E'$ lying in the contact distribution
and satisfying conditions (i), (ii) of subsection~\ref{secord:pl}.

\begin{defi} The second-order differential equation defined by the field~$E'$
  is said to be \emph{dual} to the equation corresponding to~$E$.
\end{defi}

\begin{ex} Consider the differential equation $y''=0$. The corresponding
  direction field is
$$E=\left\langle\frac\p{\p x}+z\frac\p{\p y}\right\rangle.$$
The Legendre transformation
$$(x,y,z)\mapsto (z,y-xz,-x)$$
is a contact transformation, and it takes the distribution~$E$ into the
vertical distribution~$V$, while the latter is transformed into~$E$. Thus the
equation $y''=0$ coincides with its own dual equation.
\end{ex}

Note that in general there exist several contact transformations carrying
$E$ into~$V$. Moreover, if $\psi_1$ and $\psi_2$ are two such transformations,
then the mapping $\psi_1\circ\psi_2^{-1}$ preserves the pair $(C,V)$ and hence
has the form $\v^{(1)}$ for some local diffeomorphism~$\v$ of the plane.
Therefore the dual equation is well defined only up to equivalence by local
diffeomorphisms of the plane.

In the language of pairs of direction fields, the transition to the dual
equation is equivalent to interchanging the positions of the direction fields
of the pair. A description of the dual equation can also be given in terms of
the solutions of the initial second-order equation when they are written in
the form of a two-parameter family $F(x,y,u,v)=0$, where $u,v$~ are parameters
(say, $u=y(0)$ and $v=y'(0)$). If we now consider $x$ and~$y$ as parameters,
$u$ as an independent and $v$ as a dependent variable, we obtain a 
two-parameter family of curves in the plane, which coincides with the family
of solutions of the dual second-order equation. (See \cite{arn} for details.)

\section{Cartan connections}
In this section we give only basic definitions and list the results we shall use
later. For more detail see~\cite{kob}.

\subsection{Definitions} 
Let $\G$ be a finite-dimensional Lie group, $G$ a closed subgroup of~$\G$, and 
$M_0=\G/G$ the corresponding homogeneous space of the Lie group~$\G$. Further,
let $\gb$ be the Lie algebra of~$\G$, and let $\g$ be the subalgebra of~$\gb$
corresponding to~$G$. We identify $\gb$ with the tangent space $T_e\G$.
 
Suppose $M$ is a smooth manifold of dimension 
$\dim M_0=\codim_{\gb}\g$, and $\pi\colon P\to M$ is a principal fiber bundle
with structural group~$G$ over~$M$. For $X\in\g$, let $X^*$ denote the
fundamental vector field on~$P$, corresponding to~$X$; for $g\in G$ let $R_g$
denote the right translation of~$P$ by the element~$g$:
$$R_g\colon P\to P,\ p\mapsto pg, \quad p\in P.$$

\begin{defi} A {\em Cartan connection\/} in the principal fiber bundle~$P$ is a
  1-form $\om$ on $P$ with values in~$\gb$ such that 
\begin{enumerate}[1$^\circ$.]
\item $\om(X^*)=X$ for all $X\in\g$;
\item $R_g^*\om=(\Ad g^{-1})\om$ for all $g\in G$;
\item $\om_p\colon T_pP\to\gb$ is a vector space isomorphism for all 
$p\in P$.
\end{enumerate}
\end{defi} 

\begin{ex} Suppose $M=M_0$; then $P=\G$ may be considered as a principal fiber
  bundle over~$M$ with structural group~$G$. We define $\om$ to be the
  canonical left-invariant Maurer--Cartan form on~$\G$. It is easily verified
  that $\om$ satisfies the conditions $1^\circ$--$3^\circ$ in the above
  definition, and hence $\om$ is a Cartan connection. We shall call it the 
{\em canonical Cartan connection of the homogeneous space\/} $(\G,M_0)$.
\end{ex}

\subsection{Coordinate notation} If $s_{\alpha}$ is the section of $\pi\colon
P\to M$ defined on an open subset $U_{\alpha}\subset M$, we can identify
$\pi^{-1}(U_{\alpha})$ with $U_{\alpha}\times G$ as follows:
$$\phi_{\alpha}\colon U_{\alpha}\times G\to \pi^{-1}(U_{\alpha}),\ 
(x,g)\mapsto s_{\alpha}(x)g.$$

Given some other section $s_{\beta}\colon U_{\beta}\to P$ such that the
intersection $U_{\alpha}\cap U_{\beta}$ is non-empty, we can consider the {\em
  transition function\/}~$\psi_{\alpha\beta}$, which is a function on
$U_{\alpha}\cap U_{\beta}$ with values in~$G$, uniquely
defined by
$$\phi_{\beta}\circ\phi_{\alpha}^{-1}\colon \big(U_{\alpha}\cap U_{\beta}\big)
\times G \to \big(U_{\alpha}\cap U_{\beta}\big)\times G,\ 
(x,g)\mapsto (x, \psi_{\alpha\beta}g).$$

Any principal fiber bundle~$P$ is uniquely determined by a covering
$\{U_{\alpha}\}_{\alpha\in I}$ of the manifold~$M$ with given transition
functions $\psi_{\alpha\beta}\colon U_{\alpha}\cap U_{\beta}\to G$ satisfying
the obvious condition:
$\psi_{\alpha\gamma}=\psi_{\alpha\beta}\psi_{\beta\gamma}$ for all
$\alpha,\beta,\gamma\in I$.

Let $\om$ be a Cartan connection on~$P$. The section $s_\alpha\colon
U_{\alpha}\to P$ determines the 1-form $\om_{\alpha}=s_{\alpha}^*\om$ on
$U_{\alpha}$ with values in~$\gb$. It is easy to verify that for any other
section $s_{\beta}\colon U_{\beta}\to P$ and the corresponding form
$\om_{\beta}=s_{\beta}^*\om$, the following relation holds on the intersection 
$U_{\alpha}\cap U_{\beta}$:
\begin{equation}\label{tr:f}
\om_{\beta}=\Ad\big(\psi_{\alpha\beta}^{-1}\big)\om_{\alpha}+
\psi_{\alpha\beta}^*\theta,
\end{equation}
where $\theta$ is the canonical left-invariant Maurer--Cartan form on~$G$.

\begin{rem} Formula (\ref{tr:f}) is completely identical with the corresponding
  formula for ordinary connections on principal fiber bundles. The difference
  is that in the latter case the forms $\om_{\alpha},\om_{\beta}$ have their
  values in~$\g$, not in~$\gb$.
\end{rem}

Conversely, suppose that on every submanifold $U_{\alpha}$, $\alpha\in I$,
there is given a 1-form $\om_{\alpha}$ with values in~$\gb$, and 
\begin{enumerate}[1$^\circ$.] 
\item for any $\alpha,\beta\in I$, we have (\ref{tr:f});
\item for any $\alpha\in I$ and any point $x\in U_{\alpha}$, the mapping
$$T_xM\to\gb/\g,\ v\mapsto (\om_{\alpha})_x(v)+\g $$
is a vector space isomorphism.
\end{enumerate}
It is not hard to show that there exists a unique Cartan connection~$\om$
on~$P$ such that $\om_{\alpha}=s_{\alpha}^*\om$ for all $\alpha\in I$.

\subsection{Cartan connections and ordinary connections} Let $\om$ be a Cartan
connection on~$P$. Consider the associated fiber bundle 
$\P=P\times_G\G$, where $G$ acts on~$\G$ by left shifts. Denote by $[(p,g)]$
the image of the element $(p,g)\in P\times \G$ under the natural projection
$P\times \G\to \P$. The right action
$$ [(p,g_1)]g_2=[(p,g_1g_2)],\quad [(p,g_1)]\in \P,\,g_2\in \G,$$ 
of $\G$ on $\P$ provides the fiber bundle~$\P$ with a natural structure
of principal fiber bundle over~$M$ with structural group~$\G$.
We shall identify $P$ with a subbundle in~$\P$ by means of the embedding
$$P\hookrightarrow\P,\ p\mapsto [(p,e)].$$

It can be easily shown that there exists a unique connection form~$\bom$
on~$\P$ such that $\om=\bom|_P$. Conversely, if $\bom$ is a connection form
on~$\P$ such that
\begin{equation}\label{con}
\ker \bom_p\oplus T_pP=T_p\P \quad\text{ for all } p\in P, 
\end{equation}
then the 1-form $\om=\bom|_P$ with values in~$\gb$ is a Cartan connection
on~$P$. Thus we obtain the following result:
\begin{lem}\label{pr1} Cartan connections on~$P$ are in one-to-one
  correspondence with the ordinary connections on~$\P$ whose connection forms
  satisfy condition~(\ref{con}). 
\end{lem}
 
\begin{rems}\ 

1. Condition~(\ref{con}) is equivalent to the condition that the form
$\om=\bom|_P$ defines an absolute parallelism on~$P$.

2. If $\{U_{\alpha}\}_{\alpha\in I}$ is a covering of~$M$ with given sections 
$s_{\alpha}\colon U_{\alpha}\to P$, then these sections may be extended to
sections ${\bar s}_{\alpha}\colon U_{\alpha}\to\P,\ x\mapsto
[s_{\alpha}(x),e]$, for all $\alpha\in I$, $x\in U_{\alpha}$, and it is clear
that $s_{\alpha}^*\om={\bar s}_{\alpha}^*\bom$. Thus, in terms of
``coordinates,'' the connection form~$\bom$ is defined by the same family 
$\{\om_{\alpha}\}_{\alpha\in I}$ of forms as the Cartan connections~$\om$.
\end{rems}

\subsection{Developments} 
In this subsection, by $L_g$ and~$R_g$ we denote the left and the right shifts 
of the Lie group~$\G$ by the element $g\in\G$. 

Consider the mapping
$$\gamma\colon\P\to\G/G,\quad [p,g]\mapsto g^{-1}G.$$
It is well defined, because for $h\in G$ we have
$$\gamma([ph^{-1},hg])=(hg)^{-1}G=g^{-1}G,\quad [p,g]\in \P.$$
It follows immediately from the definition that
$\gamma(\bar pg)=g^{-1}.\gamma(\bar p)$ for all $\bar p\in\P,g\in\G$.

Let $x(t)$ be an arbitrary curve in~$M$, and $u(t)$ its horizontal lift
into the fiber bundle~$\P$ by means of the connection~$\bom$
(i.e., $\bar\pi(u(t))=x(t)$ and $\bom_{u(t)}(\dot u(t))=0$ for all~$t$).
Note that $u(t)$ is unique up to the right action of~$\G$ on~$\P$.

\begin{defi} A {\em development of the curve $x(t)$ on the manifold~$M$\/} is a
  curve of the form $\tilde x(t)=\gamma(u(t))$ in the homogeneous space $\G/G$.
\end{defi}

Hence the development of a curve in~$M$ is defined uniquely up to the action
of~$\G$ on $\G/G$.

\begin{ex} Suppose $(\G,\G/G)$ is the projective space. Then a curve $x(t)$
  in~$M$ is called a geodesic if its development is a segment of a straight line.
  A similar definition is relevant in any homogeneous space that can be
  included into the projective space, for example, in affine and Euclidean
  spaces. 
\end{ex}

The above definition of development can be easily formulated for one-dimensional
submanifolds of~$M$ diffeomorphic to the line (or, locally, for any
one-dimensional submanifolds). In particular, in the above example one can
speak of one-dimensional geodesic submanifolds.
 
We shall need the following lemma.
\begin{lem}
Let $s_{\alpha}\colon U_{\alpha}\to P$ be a section of the principal fiber
bundle $P$, and let $\om_{\alpha}=s^*\om$. Assume that a smooth curve $x(t)$
lies in $U_{\alpha}$ and consider the curve $X(t)=\om_{\alpha}(\dot x(t))$ in
the Lie algebra~$\gb$. Then up to the action of $\G$ on $M_0$ the development of
$x(t)$ has the form
$\tilde x(t)=h(t)G,$
where $h(t)$ is the curve in $\G$ completely determined by the differential
equation
$$\dot h(t)=d_eL_{h(t)}(X(t)), \quad h(0)=e.$$
\end{lem}
\begin{proof}
Let $\bom$ be the connection form on the principal $\G$-bundle $\P$
corresponding to the Cartan connection~$\om$, and let $\bar s_{\alpha}\colon
U_{\alpha}\to P$ be the section generated by the section $s_{\alpha}$. Then, as
was pointed out before,
$\om_{\alpha}=s_{\alpha}^*\om=\bar s_{\alpha}^*\bom$.

Thanks to the section $\bar s_{\alpha}$, we may assume, without loss of
generality, that $\P=U_{\alpha}\times \G$ is a trivial fiber bundle, and 
$\bar s_{\alpha}(x)=(x,e)$ for all $x\in U_{\alpha}$. Then the horizontal lift
$u(t)$ of~$x(t)$ with initial condition $u(0)=(x(0),e)$ has the form
$u(t)=(x(t),g(t))$, where $g(t)$ is a curve in $\G$ such that
$\bom(\dot u(t))=0$ for all $t$. But
\begin{multline*}
\bom_{u(t)}(\dot u(t))=\bom_{(x,g)}(\dot x(t),\dot g(t))=\\
=\Ad g(t)^{-1}\circ \bom_{(x,e)}\big(\dot x(t), d_eR_{g}^{-1}(\dot g(t))\big)=\\
=\Ad g(t)^{-1}\big(\om_{\alpha}(\dot x(t))+ d_eR_{g}^{-1}(\dot g(t))\big)=0.
\end{multline*} 
Therefore, $g(t)$ is uniquely determined by the equation
$$ \dot g(t)=-d_eR_{g(t)}(X(t)),\quad g(0)=e.$$

The development of the curve~$x(t)$ has the form $\tilde x(t)=h(t)G$, where
$h(t)=g(t)^{-1}$. 
Denote by $\tau\colon \G\to \G$ the inversion in the Lie group~$\G$. Then
\begin{multline*}
\dot h(t)=d_{g(t)}\tau(\dot g(t))=d_{g(t)}\tau\circ d_eR_{g(t)} (-X(t))=\\
=d_e\big(\tau\circ R_{g(t)}\big)(-X(t))=d_e\big(L_{g(t)^{-1}}\circ \tau\big)(-X(t))=\\
=d_eL_{h(t)}(X(t)),
\end{multline*}
as was to be proved.
\end{proof}

\subsection{Curvature} By the {\em curvature form\/} of a Cartan
connection~$\om$ on~$P$ we understand the 2-form
$$\Omega=d\om+1/2[\om,\om]$$
on $P$ with values in~$\gb$.

If $\bom$ is the curvature form on~$\P$ corresponding to a Cartan
connection~$\om$ and $\overline\Omega$ is the curvature of~$\bom$, then the
form $\Omega$ is precisely the restriction of~$\overline\Omega$ to~$P$.

The definitions of Cartan connection and of the form~$\Omega$ immediately imply
that
\begin{lem}\label{pr2} The curvature form~$\Omega$ satisfies the following conditions:
\begin{enumerate}[1$^\circ$.]
\item $R_g^*\Omega=(\Ad g^{-1})\Omega$ for all $g\in G$;
\item $\Omega(X,Y)=0$ if at least one of the tangent vectors $X,Y$ is vertical.
\item (structure equation) $d\Omega=[d\om,\om]$. 
\end{enumerate}
\end{lem}

\begin{ex} It is known that the canonical Maurer--Cartan form~$\theta$ on a Lie
  group satisfies the equation $d\theta+1/2[\theta,\theta]=0$. Hence the
  canonical Cartan connection~$\om$ of an arbitrary homogeneous space
  $(\G,M_0)$ has zero curvature. Conversely, if some Cartan connection on the
  principal fiber bundle~$P$ has zero curvature, using the Frobenius theorem,
  it is easy to show that it is locally isomorphic to the canonical Cartan
  connection.
\end{ex}

We say that a Cartan connection~$\om$ is a {\em connection without torsion\/}
(or {\em connections of zero torsion\/}), if $\Omega(X,Y)\in\g$ for all tangent
vectors $X,Y$.

\begin{rem}
The {\em torsion of a Cartan connection $\om$\/} may be defined as the 2-form
$T$ on~$P$ with values in $\gb/\g$ resulting from the curvature form~$\Omega$
on passing to the quotient by~$\g$. In this case expression ``connection of
zero curvature'' has the literal meaning.
\end{rem}

\subsection{Structure function}
Given an element $X\in\gb$, let $X^*$ be the vector field on~$P$ defined by
$$(X^*)_p=(\om_p)^{-1}(X).$$ 
This definition is in agreement with the definition of the fundamental 
field $X^*$ for
$X\in\g$. Moreover, the field $X^*$ is vertical if and only if $X^*\in\g$.
Consider how $X^*$ changes along the fibers of the fiber bundle
$\pi\colon P\to G$. From the definition of Cartan connection it follows that
$$\om_{pg}\circ d_pR_g=\Ad g^{-1}\circ\om_p\quad\forall p\in P,\, g\in G.$$
Applying both sides of this equality to the tangent vector $X^*_p$ for some 
$X\in\gb$, we get
$$\om_{pg}\big(d_pR_g(X^*_p)\big)=\Ad g^{-1}(X).$$
Hence
$$\big(\Ad g^{-1}(X)\big)_{pg}^*=
d_pR_g(X^*_p)\quad\forall p\in P,\,g\in G$$
or, briefly, $dR_g(X^*)=(\Ad g^{-1}(X))^*$ $\forall g\in G$.
Thus the right shift $R_g$ transforms the fundamental vector field~$X^*$ into
the fundamental vector field corresponding to the vector $(\Ad g^{-1})X$.

\begin{defi} The {\it structure function\/} of a Cartan connection~$\om$ is
  defined as a function $c\colon P\to\mathcal L(\gb/\g\land\gb/\g,\gb)$ such
  that 
$$c(p)\colon(X+\g,Y+\g)\mapsto\Omega_p(X^*,Y^*).$$
\end{defi}

From Lemma~\ref{pr2}, $2^\circ$, we see immediately that $c(p)$ is well
defined for all $p\in P$.

Recall that the group~$G$ acts on the space $\mathcal L(\gb/\g\land\gb/\g,\gb)$
in the following natural way:
$$(g.\varphi)(X+\g,Y+\g)=(\Ad g)\varphi\big(\Ad g^{-1}(X)+\g,
\Ad g^{-1}(Y)+\g\big)$$
for all $X,Y\in\gb$.

\begin{lem} $c(pg)=(g^{-1}).c(p)$ for all $p\in P$, $g\in G$.
\end{lem}
\begin{proof} For $X,Y\in\gb$, we have
\begin{multline*}
c(pg)(X+\g,Y+\g)=\Omega_{pg}(X^*_{pg},Y^*_{pg})=\\
=\Omega_{pg}\big(d_pR_g(\Ad g(X)^*_p),d_pR_g(\Ad g(Y)_p^*)\big)=\\
=(R_g^*\Omega)_p\big(\Ad g(X)_p^*,\Ad g(Y)^*_p\big)=\\
=\Ad g^{-1}\big(\Omega_p(\Ad g(X)_p^*,\Ad g(Y)_p^*)\big)=\big(
(g^{-1}).c(p)\big)(X+\g,Y+\g),
\end{multline*}
as was to be proved.
\end{proof}

We fix a basis $\{e_1,\dots,e_{n+m}\}$ of the Lie algebra $\gb$ such that
$\{e_{n+1},\dots,e_{n+m}\}$~ form a basis of the subalgebra~$\g$. Then every
element $\varphi\in\mathcal L(\gb/\g\land\gb/\g,\gb)$ is defined by a set of
structure constants $c_{ij}^k$ ($1\leqslant i,j\leqslant n$, 
$1\leqslant k\leqslant n+m$), $c_{ij}^k=-c_{ji}^k$, where
$$\varphi(e_i+\g,e_j+\g)=\sum_{k=1}^{n+m} c_{ij}^k e_k.$$
The structure function $c\colon P\to \mathcal L(\gb/\g\land
\gb/\g,\gb)$ defines a set of functions $c_{ij}^k(p)$. These functions may be
found from the decomposition of the curvature tensor~$\Omega$ in terms of the
components of the form~$\om$. Indeed, if we write the forms $\Omega$ and $\om$
in the form
$$\Omega=\sum_{k=1}^{n+m}\Omega^k e_k,\quad
\om=\sum_{i=1}^{n+m}\om^i e_i,$$
then the forms $\om^1,\dots,\om^{n+m}$ constitute a basis for the module of
differential forms on~$P$, and hence the forms $\Omega^k$,  ($k=1,\dots,n+m$)
can be uniquely expressed in the form
\begin{equation}\label{hare}
\Omega^k=\sum_{1\leqslant i<j\leqslant n+m} c_{ij}^k\om^i\land\om^j
\end{equation}
for some smooth functions $c_{ij}^k$ on~$P$. By
Lemma~\ref{pr2},~$2^\circ$, we have $c_{ij}^k=0$ whenever $i>n$ or $j<n$.
Applying both sides of~(\ref{hare}) to the pairs of the form $(e^*_i,e^*_j)$,
$i,j=1,\dots,n$, we get at once that $c_{ij}^k=c_{ij}^k(p)$ are precisely the
coordinates of the structure function~$c$.

\section{Cartan connections and pairs of direction fields}\label{sec5}

Suppose $\G=SL(3,\R)$~, $G$ is the group of all upper-triangular matrices
in~$\G$, and $M_0=\G/G$.
Then the homogeneous space $(\G,M_0)$ admits the following two interpretations:
\begin{enumerate}
\item[1)] $M_0=J^1(\R P^2)$, the action of~$\G$ is the natural lifting of the
  standard action of $SL(3,\R)$ on $\R P^2$;
\item[2)] $M_0$ is the set of all flags
$$V_1\subset V_2\subset \R^3\quad(\dim V_i=i,\
i=1,2),$$
the action of~$\G$ is generated by the natural action of $SL(3,\R)$
on~$\R^3$.
\end{enumerate}

The corresponding pair $(\gb,\g)$ of Lie algebras has the form
$$\gb=\mathfrak{sl}(3,\R),\quad\g=
\left\{\left.\begin{pmatrix}
x_{11}&x_{12}&x_{13}\\0&x_{22}&x_{23}\\0&0&x_{33}\end{pmatrix}
\right|\ x_{11}+x_{22}+x_{33}=0\right\}.$$
Any differential form~$\om$ with values in~$\gb$ can be uniquely written in the
form
$$\om=\begin{pmatrix}
\om_{11}&\om_{12}&\om_{13}\\
\om_{21}&\om_{22}&\om_{23}\\
\om_{31}&\om_{32}&\om_{33}\end{pmatrix},
\quad \om_{11}+\om_{22}+\om_{33}=0,$$
where $\om_{ij}$~are usual differential forms.

Let $\pi\colon P\to\R^3$~be an arbitrary principal fiber bundle with structural
group~$G$, and let $\om$~be a Cartan connection on~$P$. We shall now show that
$\om$ determines a pair of direction fields in $\R^3$. Consider an arbitrary
(local) section $s\colon \R^3\to P$ and let
\begin{align*}
&E_1=\langle s^*\om_{21},s^*\om_{31}\rangle^\perp,\\
&E_2=\langle s^*\om_{31},s^*\om_{32}\rangle^\perp.
\end{align*}
Any other section $\tilde s$ has the form $\tilde s=s\v$, where $\v$ is a
smooth function from the common domain of $s$ and $\tilde s$ to~$G$.
Then
\begin{equation}\label{22eq1}
\tilde s^*\om=(\Ad\v^{-1})(s^*\om)+\v^*\theta,
\end{equation}
where $\theta$ is the Maurer--Cartan form on~$G$. Simple calculation
shows that the forms $\tilde s^*\om_{21}$ and $\tilde s^*\om_{31}$ are linear
combinations of the forms $s^*\om_{21}$ and $s^*\om_{31}$.
Hence the direction field~$E_1$ is independent of the choice of the
section~$s$. The proof that $E_2$ is also well defined is carried out in a
similar manner.
\begin{defi} Let $(E_1,E_2)$ be a pair of direction fields in space. We say
  that {\em a Cartan connection~$\om$ is associated\/} with the pair
  $(E_1,E_2)$ if $(E_1,E_2)$ coincides with the pair of direction fields
  corresponding to~$\om$.
\end{defi}

The next theorem is the central result of the present section.
\begin{thr}\label{22thr1}
If $(E_1,E_2)$~is an arbitrary nondegenerate pair of direction fields
in~$\R^3$, then up to isomorphism there exist a unique principal fiber bundle
$\pi\colon P\to\R^3$ and a unique Cartan connection~$\om$ on~$P$ satisfying the
following conditions:
\begin{enumerate}[(i)]
\item $\om$ is associated with the pair $(E_1,E_2)$;
\item for all $p\in P$, the structure function $c$ of~$\om$ lies in a
  subspace $W\subset\Lgg$ of the form
$$W=\left\{
e_{21}\land e_{31}\mapsto\left(\begin{smallmatrix}
0&a&b\\0&0&0\\0&0&0\end{smallmatrix}\right),
e_{31}\land e_{32}\mapsto\left(\begin{smallmatrix}
0&0&c\\0&0&d\\0&0&0\end{smallmatrix}\right),
e_{21}\land e_{32}\mapsto0\mid a,b,c,d\in\R\right\}.$$
(Here
$$e_{21}=\left(\begin{smallmatrix}
0&0&0\\1&0&0\\0&0&0\end{smallmatrix}\right)+\g,\quad
e_{31}=\left(\begin{smallmatrix}
0&0&0\\0&0&0\\1&0&0\end{smallmatrix}\right)+\g,\quad
e_{32}=\left(\begin{smallmatrix}
0&0&0\\0&0&0\\0&1&0\end{smallmatrix}\right)+\g$$
is a basis of the space $\gb/\g$).
\end{enumerate}
\end{thr}

\begin{rem} Condition~(ii) is equivalent to the requirement that the curvature
  form of~$\om$ have the form
$$\Om=\begin{pmatrix}
0&a\om_{21}\land\om_{31}&b\om_{21}\land \om_{31}+c\om_{31}\land\om_{32}\\
0&0&d\om_{31}\land\om_{32}\\
0&0&0\end{pmatrix}.$$
\end{rem}

\begin{rem} The space~$W$ is invariant under the action of the Lie group~$G$
  on $\Lgg$. Indeed, if
$$g=\begin{pmatrix}
x&t&v\\0&y&u\\0&0&z
\end{pmatrix}\in G,\quad xyz=1,$$
and $\v\in W$ has the form
$$\v\colon
e_{21}\land e_{31}\mapsto\left(\begin{smallmatrix}
0&a&b\\0&0&0\\0&0&0\end{smallmatrix}\right),
e_{31}\land e_{32}\mapsto\left(\begin{smallmatrix}
0&0&c\\0&0&d\\0&0&0\end{smallmatrix}\right),
e_{21}\land e_{32}\mapsto0,$$
then, by straightforward computation,
\begin{equation}\label{act}
\begin{aligned}
g.\v\colon\quad&e_{21}\land e_{31}\mapsto
\frac{x^3}{y^2z^2}\left(\begin{smallmatrix}
0&az&-au+by\\0&0&0\\0&0&0\end{smallmatrix}\right),\\
&e_{31}\land e_{32}\mapsto
\frac{xy}{z^3}\left(\begin{smallmatrix}
0&0&cx+dt\\0&0&dy\\0&0&0\end{smallmatrix}\right),\\
&e_{21}\land e_{32}\mapsto0.
\end{aligned}
\end{equation}

In particular, $g.\v\in W$.
\end{rem}

\begin{proof} We shall first consider the local case and then the global one.
According to Theorem~\ref{thr1}, we can always assume that locally
$P=\R^3\times G$ is a trivial principal
fiber bundle and that the direction fields $E_1$ and~$E_2$ have the form
$$E_1=\left\langle\frac\p{\p z}\right\rangle,\quad
E_2=\left\langle\frac\p{\p x}+z\frac\p{\p y}+f(x,y,z)\frac\p{\p z}\right\rangle,$$
where $f$~is a smooth function on~$\R^3$.

Let $s\colon \R^3\to P$, $a\mapsto(a,e)$~be the trivial section of the fiber
bundle~$\pi$, and $\tom=s^*\om$~a form on~$\R^3$ with values in~$\gb$. Note
that the form~$\tom$ uniquely determines the form~$\om$. Indeed, for points
of the form $p=(a,e)\in P$ ($a\in\R$), we have
$$(\om_p)|_{T_a\R^3}=\tom_a,\quad
(\om_p)|_{T_e G}=\id_{\g},$$
and from the definition of Cartan connection it follows that for any point $p=(a,g)$,
$$\om_p=(\Ad g)\circ\om_{(a,e)}\circ d_{(a,g)}R_{g^{-1}}.$$

Let $\TOm=s^*\Om$. Then, in just the same way, the form $\TOm$ uniquely
determines the curvature form~$\Om$.

Condition~(i) on the Cartan connection~$\om$ is equivalent to the following
conditions:
\begin{align*}
(\tilde{i})\qquad&\tom_{21}=\a dx+\delta(dy-z\,dx);\\
&\tom_{31}=\gamma(dy-z\,dx);\\
&\tom_{32}=\lambda(dy-z\,dx)+\mu(dz-f\,dx).
\end{align*}
Since the subspace $W\subset\Lgg$ is $G$-invariant, condition~(2) needs to be
verified only for the points $p=s(a)$ for all $a\in\R^3$.
Thus condition~(ii) is equivalent to the condition
$$(\widetilde{ii})\qquad\TOm=\begin{pmatrix}
0&\tilde a\om_{21}\land\om_{31}&\tilde b\om_{21}\land \om_{31}+
\tilde c\om_{31}\land\om_{32}\\
0&0&\tilde d\om_{31}\land\om_{32}\\
0&0&0\end{pmatrix}.$$
Show that there is a unique connection~$\om$ on~$P$ satisfying these
conditions.

Note that the identification $P\equiv\R^3\times G$ is not canonical and that
all such identification are in one-to-one correspondence with the sections
$\tilde s\colon\R^3\to P$ of the form $s'=s\v$, where $\v\colon
\R^3\to G$ is an arbitrary smooth function. Under an identification like this,
the form~$\tom$ becomes
$$\tom'=(s')^*\om=(\Ad\v^{-1})\tom+\v^*\theta,$$
where $\theta$ is the canonical Maurer--Cartan form on~$G$.

With a suitable choice of~$s$, we can always ensure that
\begin{align*}
&\tom_{21}=dx,\\
&\tom_{31}=dy-z\,dx,\\
&\tom_{32}=k(dz-f\,dx).
\end{align*}
This determines~$s$ uniquely up to a transformation of the form $s\mapsto s\v$,
where $\v\colon\R^3\to G_1$ is a transition function with values in the
subgroup
$$G\supset G_1=\left\{\left.\begin{pmatrix}
1&0&x\\0&1&0\\0&0&1\end{pmatrix} \right|\ x\in\R\right\}.$$ Using the condition
$\TOm_{31}=0$, we find that $k=1$ and $\tom_{33}-\tom_{11}=u(dy-z\,dx)$ for
some smooth function~$u$ on~$\R^3$.

The condition $\TOm_{21}=0$ implies that
\begin{align*}
\tom_{22}-\tom_{11}=&v\,dx+t(dy-z\,dx),\\
\tom_{23}=&t\,dx+w(dy-z\,dx),
\end{align*}
where $v,w,t\in C^\infty(\R^3)$. By a suitable choice of the section~$s$, the
number~$t$ can always be made equal to zero, and thus $s$ is uniquely defined.

Now from the condition $\TOm_{32}=0$ we obtain that $v=\frac{\p f}{\p z}$ and
$$\tom_{12}=\frac{\p f}{\p y}dx-u(dz-f\,dx)+\lambda(dy-z\,dx),\quad
\lambda\in C^\infty(\R^3).$$

Furthermore, the equality $\TOm_{22}+\TOm_{33}-2\TOm_{11}=0$ implies that
$u=-\frac12\frac{\p^2f}{\p z^2}$ and that
$$\tom_{13}=\lambda dx-\frac13\frac{\p^2 f}{\p y\p z}dx-\frac16
d\left(\frac{\p^2 f}{\p z^2}\right)+\mu(dy-z\,dx),\quad
\mu\in C^\infty(\R^3).$$

Similarly, from $\TOm_{22}-\TOm_{11}=0$ it follows that
\begin{align*}
&w=\frac16\frac{\p^3 f}{\p z^3},\\
&\lambda=\frac23\frac{\p ^2f}{\p y\p z}-\frac16
\frac d{dx}\left(\frac{\p^2 f}{\p z^2}\right),
\end{align*}
where $\frac d{dx}$ denotes the vector field $\frac\p{\p x}+z\frac\p{\p y}+
f\frac\p{\p z}$.

Thus, if
\begin{equation}\label{22eq2}
\TOm=\begin{pmatrix}
0&*&*\\0&0&*\\0&0&0\end{pmatrix},
\end{equation}
then
\begin{gather*}
\tom_{21}=dx,\quad\tom_{31}=dy-z\,dx,\quad\tom_{32}=dz-f\,dx,\\
\tom_{22}-\tom_{11}=\frac{\p f}{\p z}dx,\quad
\tom_{33}-\tom_{11}=-\frac12\frac{\p^2f}{\p z^2}(dy-z\,dx),\\
\begin{split}
\tom_{12}&=\frac{\p f}{\p y}dx+\frac12\frac{\p^2 f}{\p z^2}(dz-f\,dx)+
\left(\frac23\frac{\p^2f}{\p y\p z}-\frac16\frac d{dx}\frac{\p^2f}{\p z^2}\right)(dy-z\,dx),\\
\tom_{13}&=\left(\frac13\frac{\p^2f}{\p y\p z}-\frac16\frac d{dx}\frac{\p^2f}{\p z^2}
\right)dx-\frac16d\left(\frac{\p^2f}{\p z^2}\right)+\mu(dy-z\,dx),\\
\tom_{23}&=\frac16\frac{\p^3f}{\p z^3}(dy-z\,dx).
\end{split}
\end{gather*}
The only arbitrary coefficient here is~$\mu$.

From~(\ref{22eq2}), taking into account the structure equation of curvature
from Lemma~\ref{pr2}, we can obtain the following conditions:
\begin{gather*}
\tom_{31}\land\TOm_{23}=0,\quad\tom_{31}\land\TOm_{12}=0,\\
\tom_{21}\land\TOm_{12}+\tom_{31}\land\TOm_{13}=0,\quad
\tom_{32}\land\TOm_{23}-\tom_{21}\land\TOm_{12}=0,
\end{gather*}
whence
\begin{align*}
&\TOm_{23}=\tilde d\tom_{31}\land\tom_{32}+\tilde e\tom_{21}\land\tom_{31},\\
&\TOm_{12}=\tilde a\tom_{21}\land\tom_{31}+\tilde e\tom_{31}\land\tom_{32},\\
&\TOm_{13}=\tilde b\tom_{21}\land\tom_{31}+\tilde c\tom_{31}\land\tom_{32}+
\tilde e\tom_{21}\land\tom_{32}
\end{align*}
for some $\tilde a,\tilde b,\tilde c,\tilde d,\tilde e\in C^{\infty}(\R^3)$.

Finally, letting
$$\mu=\frac16\frac{\p^3}{\p y\p z^2}-\frac16\frac{\p f}{\p z}\cdot\frac{\p^3f}
{\p z^3}-\frac16\frac d{dx}\frac{\p^3f}{\p z^3},$$
we can ensure that $e=0$.

\smallskip
We now show that in the global case the fiber bundle $\pi\colon P\to\R^3$ and
the Cartan connection~$\om$ are, too, unique up to isomorphism. For this
purpose we take a covering $\{U_\a\}_{\a\in I}$ of the three-dimensional
space and, using local consideration from the first part of the proof,
construct a Cartan connection $\om_\a$ on each trivial fiber bundle
$\pi_\a\colon U_\a\times G\to U_\a$.

Let $s_\a,s_\b$~be the trivial sections of the fiber bundles $\pi_\a$ and
$\pi_\b$, respectively, and let $\tom_\a=s^*_\a\om_\a$, $\tom_\b=s^*_\b\om_\b$.
Show that for any two disjoint domains $U_\a,U_\b$ there exists a unique
function
$$\v_{\a\b}\colon U_\a\cap U_\b\to G$$
such that
\begin{equation}\label{22eq3}
\tom_\b=(\Ad\v_{\a\b}^{-1})\tom_\a+\v_{\a\b}^*\theta.
\end{equation}
Without loss of generality we can assume that the form $\tom_\a$ has the form
stated above. The form $\tom_\b$ on $U_\a\cap U_\b$ also satisfies
conditions (1), (2) and, therefore, can be made $\tom_\a$ by a suitable choice
of section. But, by the above, this section is uniquely defined.
Thus there exists a unique function $\v_{\a\b}\colon U_\a\cap U_\b\to G$
satisfying~(\ref{22eq3}).

In its turn, the family of functions $\{\v_{\a\b}\}_{\a,\b\in I}$ uniquely (up
to isomorphism) determines a principal fiber bundle $\pi\colon P\to \R^3$ with
structural group~$G$ and a Cartan connection~$\om$ on~$P$.
\end{proof}

Theorem~\ref{22thr1} implies that there is a one-to-one correspondence between
second-order differential equations and the Cartan connections~$\om$ on
principal fiber bundles $\pi\colon P\to J^1(\R^2)$ with structural
group~$G$ satisfying condition~(ii). In particular, two second-order equations
are (locally) equivalent if and only if so are the corresponding Cartan
connections.

Explicit calculations show that the functions $\tilde a$ and $\tilde d$ for the
constructed form~$\tom$ on $\R^3$ have the form
\begin{align*}
\tilde a&=-f_{yy}+\frac12ff_{yzz}+\frac12f_yf_{zz}+\frac23f_{xyz}-\frac16f_{xxzz}-\\
 &-\frac13zf_{xyzz}-\frac16f_xf_{zzz}-\frac13ff_{xzzz}+\frac23zf_{yyz}-\\
 &-\frac16z^2f_{yyzz}-\frac16zf_yf_{zzz}-\frac13zff_{yzzz}-\frac23f_zf_{yz}+\\
 &+\frac16f_zf_{xzz}+\frac16zf_zf_{yzz}-\frac16f^2f_{zzzz};\\[2mm]
\tilde d&=-\frac16f_{zzzz}.
\end{align*}
Further, from the structure equation $d\Om=[\Om,\om]$ it follows that
$$ \tilde b=\frac{\p \tilde a}{\p z},\quad \tilde c=-\frac{d}{dx}(\tilde
d)-2f_z\tilde d.$$
In particular, the conditions $\tilde a=0$ and $\tilde d=0$ imply that $\tilde b=0$ and
$\tilde c=0$, respectively.

From formulas~(\ref{act}) of the action of the Lie group~$G$ on the space~$W$,
in which the structure function takes its values, it follows that the
equalities $\tilde a=0$ and $\tilde d=0$ have invariant nature. To be precise, the subspaces
$W_1$ and~$W_2$ in~$W$ defined by the equalities $a=b=0$ and $c=d=0$,
respectively, are also $G$-invariant. In particular, the classes of
second-order equations described by these equalities are stable under the
diffeomorphism group of the plane.

The condition $\tilde c=\tilde d=0$ is equivalent to the condition that the corresponding
differential equation has the form
\begin{equation}\label{ekub}
y''=A(y')^3+B(y')^2+Cy'+D,
\end{equation}
where $A,B,C,D$ are functions of $x,y$.

The condition $\tilde a=\tilde b=0$ implies that the dual differential equation has the
form~(\ref{ekub}). Later on we shall show that the whole structure function is
zero if and only if that the corresponding differential equation is locally
equivalent to the equation $y''=0$. Thus,
\begin{thr}
  A second-order differential equation is equivalent to the equation $y''=0$ if
  and only if both the equation itself and its dual have the form~(\ref{ekub}).
\end{thr}

The main results of this section are due to \`E.~Cartan~\cite{car}. Modern
versions of these results from slightly different points of view can be found
 in \cite{fels,gat,gl,gris,olver2,rom}.

\section{Projective connections in the plane and 
equations of degree~3 with respect to~$y'$}

As shown in the previous section, the class of second-order equations of the form
\begin{equation}\label{e23}
y''=A(y')^3+B(y')^2+Cy'+D
\end{equation}
is stable under the diffeomorphism group of the plane. In this section we show
that every equation of this form may be naturally associated with some
projective Cartan connection on the plane.

Let $M_0=\R P^3$, and let $\G=SL(3,\R)$ be the group of projective
transformations of~$M_0$. Fix the point $o=[1:0:0]$ in~$M_0$. The stationary
subgroup $G=\G_o$ has the form
$$G=\left\{\left.\begin{pmatrix} (\det A)^{-1}&B\\0&A\end{pmatrix}\right|
\,A\in GL(2,\R), B\in \Mat_{1\times2}(\R)\right\},$$
and the corresponding pair $(\gb,\g)$ of Lie algebras has the form
$$\gb=\mathfrak{sl}(3,\R), \g=\left\{\left.\begin{pmatrix} -\tr A&B\\0& A\end{pmatrix}
\right|\,A\in \mathfrak{gl}(2,\R), B\in \Mat_{1\times 2}(\R)\right\}.$$

\begin{defi}
A {\it projective connection\/} in the plane is a Cartan connection with the
model $(\G,M_0)$.
\end{defi}

Any differential form~$\om$ with values in the Lie algebra~$\gb$ may be
expressed uniquely in the form
$$\om=\begin{pmatrix} -\om_1^1-\om^2_2 &\om_1&\om_2\\ \om^1&\om_1^1&\om_2^1\\ 
\om^2&\om_1^2&\om_2^2\end{pmatrix}.$$ 
Suppose $\pi\colon P\to \R^2$ is a principal fiber bundle with structural
group~$G$, and $\om$ is a Cartan connections on~$P$, and let $(x,y)$ be a
coordinate system in the plane.

\begin{lem}
Locally, in the neighborhood of any point $(x,y)\in \R^2$ there exists a unique
section $s\colon\R^2\to P$ such that
$$s^*\om=\begin{pmatrix}  0 & * & * \\ dx & * & * \\ dy & * & * \end{pmatrix}.$$
\end{lem}
\begin{proof}
Let $s$ be an arbitrary local section. Then $s$ is defined up to a
transformation $s\to s\v$, where $\v\colon\R^2\to G$ is a smooth transition
function. If
$$(s^*\om^1)=\a\d x+\b\d y,$$
$$(s^*\om^2)=\gamma\d x+\delta\d y,$$
then, choosing $\v$ to have the form
$$\v\colon(x,y)\mapsto\begin{pmatrix} \l^{-2}(\det A)^{-1}&0\\ 0 &\l A\end{pmatrix},$$
where
$$A=\begin{pmatrix} \a&\b\\ \gamma&\delta\end{pmatrix}, \l=(\det A)^{-1/3},$$
we get
$$s^*\om=\begin{pmatrix} {}* & * & *\\ dx & * & *\\ dy & * & *\end{pmatrix}.$$
Then the section $s$ is defined uniquely up to transformations of the form
$s\mapsto s\v$, where $\v\colon\R^2\to G_1$ takes values in the subgroup
$$G\supset G_1=\left\{\left.\begin{pmatrix} 1&B\\ 0&E_2\end{pmatrix}\right|\, 
B\in \Mat_{1\times 2}(\R)\right\}.$$
Similarly, it is easy to show that there is a unique function~$\v$ such that
the section $s\v$ has the desired form.
\end{proof}

Recall that a {\it geodesic\/} of a projective connection~$\om$ is a curve in
the plane whose development is a segment of a straight line in $\R P^2$.
Similarly, a {\it geodesic submanifold\/} is a one-dimensional submanifold in
the plane whose development is a segment of a straight line in~$\R P^2$
irrespective of the parametrization. Suppose that a submanifold~$L$ in the
plane is the graph of some function~$y(x)$. Then $L$ may be parametrized like
this: 
$$t\mapsto x(t)=(t,y(t)), \quad t\in\R.$$
We shall now determine when $L$ is a geodesic submanifold.

Consider the section $s\colon\R^2\to P$ satisfying the conditions of the above
lemma, and let $\tom=s^*\om$. Define a curve~$X(t)$ in the Lie algebra $\gb$ as
$$X(t)=\tom(\dot x(t))=
\begin{pmatrix} 0 & \tom_1(\dot x)&\tom_2(\dot x)\\ 
1  &\tom_1^1(\dot x)&\tom_2^1(\dot x)\\ 
y' &\tom_1^2(\dot x)&\tom_2^2(\dot x)\\
\end{pmatrix}.$$
Then the development of the curve $x(t)$ has the form $\tilde x(t)= h(t).o$, where
$h(t)$ is the curve in $SL(3,\R)$ such that
$$\dot h(t)=h(t)X(t),\quad h(0)=E_3.$$
Passing to the section $s'=s\v$, where
$$\v\colon(x,y)\mapsto \begin{pmatrix} 1&0&0\\0&1&0\\ 0&y'(x)&1\end{pmatrix},$$
we get
$$\tilde \om'=(s')^*\om=(\Ad \v^{-1})\tom+\v^{-1}\d \v.$$
Straightforward computation shows that
$$X'(t)=\tom'(\dot x(t))=
\begin{pmatrix} 0&*&*\\ 1&*&*\\
0 & y''+(\tom_1^2+y'(\tom_2^2-\tom_1^1)-(y')^2\tom_2^1)
& *\end{pmatrix}.$$
It follows that the tangent vector to the development $\tilde x(t)$ at the
point~$o$ in non-homogeneous coordinates is equal to $(1,0)$. 
But in $\R P^2$ there exists a unique straight
line~$l$ through~$o$ with tangent vector $(1,0)$, namely $l=\{[x:y:0]\}$.
Hence a necessary and sufficient condition for $L$ to be a geodesic submanifold
is that $\tilde x(t)\in l$ for all $t\in \R$. This condition is equivalent to
the requirement that $h(t)$ lie in the subgroup $H\subset \G$ preserving the
straight line~$l$:
$$H=\left\{\left.\begin{pmatrix} a&b&c\\d&e&f\\0&0&g\end{pmatrix}\right|\,
g(ae-bd)=1\right\}.$$
But this is possible if and only if the curve $X'(t)$ lies in the Lie algebra
of~$H$, which has the form
$$\h=\left\{\left.\begin{pmatrix} a&b&c\\d&e&f\\0&0&g\end{pmatrix}\right|\,
a+e+g=0\right\}.$$
Thus every geodesic submanifold satisfies the equation
\begin{equation}\label{nonexp}
y''=-(\tom_1^2+y'(\tom_2^2-\tom_1^1)-(y')^2\tom_2^1)
(\dot x).
\end{equation}
Assume that
\begin{align*}
&\tom_1^1=-\tom_2^2=\a_1^1\d x+\b_1^1\d y,\\
&\tom_2^1=\a_2^1\d x+\b_2^1\d y,\\
&\tom_1^2=\a_1^2\d x+\b_1^2\d y.
\end{align*}
Then equation~(\ref{nonexp}) can be written explicitly as
$$y''=\b_2^1(y')^3+(2\b_1^1+\a_2^1)(y')^2+(2\a_1^1-\b_1^2)y'-\a_1^2.$$

Any second-order equation of the form~(\ref{e23}) can therefore be interpreted
as an equation for geodesic submanifolds of some projective connection in the
plane. 

We now formulate the main result of this section.
\begin{thr}
Given a second-order differential equation of the form~(\ref{e23}), there exists
a unique (up to isomorphism) principal fiber bundle $\pi\colon P\to \R^2$ with
structural group~$G$ and a unique Cartan connection~$\om$ on~$P$ satisfying the
following conditions:
\begin{enumerate}[(i)]
\item the geodesics of~$\om$ satisfy the equation~(\ref{e23});
\item the structure function $c\colon P\to\Lgg$ takes values in the 
  subspace 
\end{enumerate}
$$W=\left\{e_1\land e_2\mapsto \left(\begin{smallmatrix}
\ 0&a&b\\0&0&0\\ 0&0&0\end{smallmatrix}\right)
a,b\in \R\right\}.$$
(Here $e_1=\left(\begin{smallmatrix} 0&0&0\\1&0&0\\0&0&0\end{smallmatrix}\right)+\g,$
$e_2=\left(\begin{smallmatrix} 0&0&0\\0&0&0\\1&0&0\end{smallmatrix}\right)+\g$
is a basis of the quotient space $\gb/\g$).
\end{thr}

\begin{rem}
  It is easy to show that the subspace~$W$ is invariant under the action of~$G$
  on~$\Lgg$.
\end{rem}

\begin{proof}
The proof of this theorem is very similar to that of Theorem~\ref{22thr1},
Section~\ref{sec5}. We shall only construct the $\gb$-valued form $\tom=s^*\om$
on~$\R^2$. We can assume, without loss of generality, that 
\begin{align*}
&\tom^1=\d x,\\
&\tom^2=\d y,\\
&\tom_1^1+\tom_2^2=0.
\end{align*}
Suppose
\begin{align*}
&\tom_1^1=-\tom_2^2=\a_1^1\d x+\b_1^1\d y,\\
&\tom_2^1=\a_2^1\d x+\b_2^1\d y,\\
&\tom_1^2=\a_1^2\d x+\b_1^2\d y.
\end{align*}
Then condition~(i) together with the equations $\TOm^1=
\TOm^2=0$ gives the following system of equations for the coefficients 
$\a_j^i,\b_j^i$ ($i,j=1,2$):
$$\begin{cases}
\b_2^1=A,\\
2\b_1^1+\a_2^1=B,\\
2\a_1^1-\b_1^2=C,\\
-\a_1^2=D;
\end{cases}
\quad
\begin{cases}
\b_1^1=\a_2^1,\\
\b_1^2=-\a_1^1.
\end{cases}$$
It immediately follows that
\begin{align*}
&\tom_1^1=-\tom_2^2=\frac13(C\,dx+B\,dy),\\
&\tom_2^1=\frac13B\,dx+A\,dy,\\
&\tom_1^2=-D\,dx-\frac13C\,dy.
\end{align*}
The conditions $\TOm_j^i=0$ ($i,j=1,2$) uniquely determine the forms
\begin{align*}
&\tom_1=\left(\frac{\p D}{\p y}-\frac13\frac{\p C}{\p x}-\frac23BD
+\frac29C^2\right)\d x+
\left(\frac13\frac{\p C}{\p y}-\frac13\frac{\p B}{\p x}
+\frac19B^2-AD\right)\d y,\\
&\tom_2=\left(\frac13\frac{\p C}{\p y}-\frac13\frac{\p B}{\p x}
+\frac19BC-AD\right)\d x+
\left(\frac13\frac{\p B}{\p y}-\frac{\p A}{\p x}+\frac29B^2-\frac23AC\right)\d y.
\end{align*}
\end{proof}

\section{Invariants of Cartan connections}
\subsection{Absolute parallelism}

Cartan connections constitute a special case of a more general concept, the
concept of {\it absolute parallelism.\/}

\begin{defi} An
  {\it absolute parallelism\/} on a manifold~$M$ is an ordered set
  of vector fields (a \emph{frame}) $\{X_1,\dots,X_m\}$ ($m=\dim M$)
  that form a basis of the tangent space $T_xM$ at each point $x\in M$.
\end{defi}

To any set $\{X_1,\dots,X_m\}$ of this kind we can assign a set of differential
1-forms (a {\em coframe}) $\{\om_1,\dots,\om_m\}$, uniquely defined by the
condition $\om_i(X_j)=\delta_{ij}$ for all $i,j=1,\dots,m$. Then at each point 
$x\in M$, the forms $\{\om_1,\dots,\om_m\}$ form a basis of the dual space
$T_x^*M$.

\begin{ex}
Let $\pi\colon P\to M$ be a principal fiber bundle with structural group~$G$,
and $\om\colon TP\to \gb$ a certain Cartan connection. If
$\{e_1,\dots,e_{n+m}\}$ is an arbitrary basis of $\gb$, the corresponding
fundamental vector fields $\{e_1^*,\dots,e_{n+m}^*\}$ define an absolute
parallelism on~$P$. The dual coframe is made up 
precisely of the forms $\om_1,\dots,\om_{n+m}$ in the decomposition
$$\om=\om_1e_1+\dots+\om_{n+m}e_{n+m}.$$
\end{ex}

Let $\{X_1,\dots,X_m\}$ be an absolute parallelism on~$M$ with dual
coframe $\{\om_1,\dots,\om_m\}$. The functions $c_{ij}^k$ on~$M$ defined by 
$$[X_i,X_j]=-\sum_{k=1}^m c_{ij}^k X_k$$
are called the {\it structure functions\/} of the structure $\{
X_1,\dots,X_m\}$. These functions can also be found from the equalities
$$\d \om_k=\sum_{1\le i<j\le m} c_{ij}^k \om_i\land\om_j.$$

If $\{\TX_1,\dots,\TX_m\}$ is an absolute parallelism on a
manifold~$\widetilde M$ and $\v\colon  M\to \widetilde M$ is a local
diffeomorphism such that $\d\v(X_i)=\TX_i$ for all $i=1,\dots,m$, then
obviously $c_{ij}^k=\tilde c_{ij}^k\circ\v$ for all $i,j,k=1,\dots,m$. Hence
the structure functions $c_{ij}^k$ are {\it invariants\/} of the frame
$\{X_1,\dots,X_m\}$.

New invariants of the frame $\{X_1,\dots,X_m\}$ may be derived from the
structure functions by means of {\it covariant differentiation\/}. The {\em
  covariant derivative\/} of a scalar (or even vector-valued) function~$f$
along $X_i$ ($i=1,\dots,m$) is defined to be $X_if$. It is clear that if $f$ is
an invariant of the frame $\{X_1,\dots,X_m\}$, then so is the function 
for all $i=1,\dots,m$. The covariant derivatives of a function~$f$ can also be
defined with the help of the differential forms $\{\om_1,\dots,\om_m\}$.
Indeed, the differential 1-form $df$ can be uniquely expressed in the form
$$\d f=f_1\om_1+\dots +f_n\om_n.$$
Then the coefficient $f_i=\d f(X_i)=X_i f$ is precisely the covariant
derivative of~$f$ along~$X_i$.

We shall call the structure functions $c_{ij}^k$ the {\em first-order
  invariants} of the structure $\{X_1,\dots,X_m\}$, and assume that every
application of the operation of covariant differentiation increases the order
of the invariant by~1. Thus the most general invariant of order $s+1$ of the
frame $\{X_1,\dots,X_m\}$ has the form
$$X_{i_1}X_{i_2}\dots X_{i_s}(c_{ij}^k),\ 1\le i<j\le m;\
i_1,\dots,i_s,k=1,\dots,m.$$ 
Let us denote this invariant by $f_{\alpha}$ where $\alpha=(i_1\dots i_s[ijk])$
is a multiindex, and $|\alpha|=s+1$.

Suppose $f$ is an invariant of order~$s$. Observe that, although two vector
fields $X_i$ and $X_j$ $(1\le i<j\le m)$, generally speaking, do not commute,
the invariants of the form $X_i(X_jf)$ and $X_j(X_if)$ $(1\le i<j\le m)$
are always dependent modulo the invariants of order $s+1$:
$$X_i(X_jf)-X_j(X_if)=[X_i,X_j]f=\sum_{k=1}^m c_{ij}^k(X_kf).$$

\begin{defi} The maximal number of functionally independent invariants of an
  absolute parallelism is called the \emph{rank} of this
  structure. 
\end{defi}

If $\{X_1,\dots,X_m\}$ is an absolute parallelism and if $r_s$ is
the maximal number of its functionally independent invariants of 
order~$\le s$, then the sequence $r_1,r_2,\dots$ is obviously non-decreasing
and bounded by the dimension of~$M$. Let $N$ be the smallest natural number
such that $r_N=r_{N+1}$. It turns out that the rank of the absolute parallelism
$\{X_1,\dots,X_m\}$ is equal precisely to~$r_N$, and the number~$N$ is said to
be the order of this structure. Let $\{f_1,\dots,f_r\}$ be a set of
functionally independent invariants of order $\le N$. Then, locally, any other
invariant~$f_\a$ may be expressed uniquely in the form
$$f_\a=F_\a(f_1,\dots,f_r).$$
It can be shown that an absolute parallelism is uniquely (up to
local equivalence) determined by the functions~$F_\a$ corresponding to all
invariants~$f_\a$ of order $\le N+1$.

Using more invariant language, this can be described as follows. We define a
{\em structure mapping of order~$s$} to be the mapping
$$C^{(s)}\colon M\to\R^{d(s)},~d(s)=\frac{m^2(m^s-1)}2$$
whose components are precisely all invariants of order~$\le s$.

Then $r_s$ is precisely the dimension of the image of the mapping~$C^{(s)}$ at
its regular points. The absolute parallelism of order~$N$ is then
uniquely (up to local equivalence) determined by the image of~$C^{(N+1)}$.

\begin{defi}
  A local diffeomorphism~$\phi$ of the manifold~$M$ is called a {\em local
  automorphism\/} of the frame $\{X_1,\dots,X_m\}$ if $\phi_*X_i=X_i$ for all
  $i=1,\dots,m$.

  An {\em infinitesimal automorphism\/} of the frame $\{X_1,\dots,X_m\}$ is a
  vector field~$Y$ on~$M$ such that the local one-parameter group of
  diffeomorphisms generated by this field consists of local automorphisms of the
  frame.
\end{defi}

It is obvious that for~$Y$ to be an infinitesimal symmetry of the frame
$\{X_1,\dots,X_m\}$, it is necessary and sufficient that $[Y,X_i]=0$ for all
$i=1,\dots,m$.

It should also be pointed out that the automorphism group of the 
absolute parallelism of rank~$r$ is a Lie group of dimension~$m-r$.
Its action on $M$ is free and the
orbits of this action are precisely the inverse images of the points of 
$C^{(N)}(M)$ under the structure mapping $C^{(N)}$.

The set of all infinitesimal symmetries of the frame 
$\E=\{X_1,\dots,X_m\}$ forms a finite-dimensional Lie algebra, denoted
$\sym(\E)$, whose dimension is equal precisely to $m-r$, where $r$ denotes the
rank of~$\E$. Moreover, the restriction of the frame $\{X_1,\dots,X_m\}$ to any
of orbits of the symmetry group forms a Lie algebra isomorphic 
with the symmetry algebra of our absolute parallelism.

For complete proofs and detail see \cite{olver2,stern}.

\begin{lem}
  If $f$ is an invariant of the frame $\E=\{X_1,\dots,X_m\}$ and
  $Y\in\sym(\E)$, then $Yf=0$.
\end{lem}

\begin{proof}
Suppose
  $$[X_i,X_j]=-\sum_{k=1}^mc_{ij}^kX_k,$$
where $c_{ij}^k$ are the first-order invariants of the frame~$\E$. Then for 
any $i,j=1,\dots,m$ we have
$$[Y,[X_i,X_j]]=[[Y,X_i],X_j]+[X_i,[Y,X_j]]=0.$$
On the other hand,
$$[Y,-\sum_{k=1}^mc_{ij}^kX_k]=-\sum_{k=1}^m(c_{ij}^k[Y,X_k]+Y(c_{ij}^k)X_k)=
  -\sum_{k=1}^m Y(c_{ij}^k)X_k.$$
It follows immediately that $Y(c_{ij}^k)=0$ for all $k=1,\dots,m$. Suppose
  further that $Yf=0$ for any invariant of order $\leq n$. Every invariant of
  order~$n+1$ has the form $X_if$ for some $1\leq i\leq m$ and some
  invariant~$f$ of order~$n$. But
$$Y(X_if)=X_i(Yf)+[Y,X_i]f=0,$$
which completes the proof of the lemma.
\end{proof}

Now let $\E=\{X_1,\dots,X_m\}$ and $\ov\E=\{\ov X_1,\dots,\ov X_m\}$ be two
absolute parallelisms on manifolds $M$ and~$\ov M$, respectively,
and suppose that
\begin{itemize}
\item[a)] the frames $\E$ and $\ov\E$ have the same rank~$r$ and the same
  order~$s$;
\item[b)] if $f_{\alpha_1},\dots,f_{\alpha_r}$ is a collection of 
functionally independent invariants of order $\leq s$ of the frame~$\E$, then
the corresponding invariants $\ov f_{\alpha_1},\dots,\ov f_{\alpha_r}$ of
the frame~$\ov\E$ are also functionally independent;
\item[c)] for any two corresponding invariants $f_\beta$ and $\ov
  f_\beta$ of the frames $\E$ and $\ov\E$, whose order is $\leq s+1$, the
  functions $F_\beta$ and~$\ov F_\beta$ defined by
  \begin{align*}
    f_\beta&=F_\beta(f_{\alpha_1},\dots,f_{\alpha_r}),\\
    \ov f_\beta&=\ov F_\beta(\ov f_{\alpha_1},\dots,\ov f_{\alpha_r}),
  \end{align*}
\end{itemize}
coincide.

As it was stated above, these three conditions are necessary and sufficient in
order that the frames $\E$ and~$\ov\E$ be locally equivalent.

Given two locally equivalent frames $\E$ and~$\ov\E$, let us look at the
problem of finding a diffeomorphism $\varphi\colon M\to\ov M$ establishing the
equivalence. For this purpose, consider the set
$$L=\{\,(p,\ov p)\in M\times\ov M\mid f_{\alpha_i}(p)=\ov f_{\alpha_i}(\ov p),~
i=1,\dots,r\,\}.$$ 
It is easy to show that
\begin{enumerate}
\item $L$ is a submanifold in $M\times\ov M$;
\item the vector fields $X_i-\ov X_i$, $i=1,\dots,m$, are tangent to~$L$, and
  their restrictions to~$L$ generate a completely integrable distribution~$E$
  of dimension~$m$ on~$M$;
\item locally maximal integral manifolds of~$E$ are precisely the graphs of
  local diffeomorphisms~$\varphi\colon M\to\ov M$ establishing the equivalence
  of $\E$ and~$\ov\E$.
\end{enumerate}

Below we use the terminology and results from Appendix~A.

\begin{lem}
  The Lie algebra $\sym(\E)$ is tangent to~$L$, and its restriction to~$L$ is a
  simply transitive symmetry algebra of the distribution~$E$.
\end{lem}

\begin{proof}
  Indeed, if $Y\in\sym(\E)$, we have
$$Y(f_{\alpha_i}-\ov f_{\alpha_i})=Y(f_{\alpha_i})-Y(\ov f_{\alpha_i})=0, $$
so that $Y$ is tangent to~$L$.

Now since
$$[Y,X_i-\ov X_i]=[Y,X_i]-[Y,\ov X_i]=0,$$
we see that the restriction of~$Y$ to~$L$ is a symmetry of the
distribution~$E$.

Suppose that at some point $(p,\ov p)\in L$, we have $Y_{(p,\ov p)}\in
E_{(p,\ov p)}$. Then
$$Y_{(p,\ov p)}=\sum_{i=1}^m\alpha_i\bigl((X_i)_p-(\ov X_i)_{\ov p}\bigr)=
  \sum_{i=1}^m\alpha_i(X_i)_p-\sum_{i=1}^m\alpha_i(\ov X_i)_{\ov p},$$
$\alpha_1,\dots,\alpha_m\in\R$ being constants. The projection of the tangent
  vector $Y_{(p,\ov p)}$ to $T_{\ov p}\ov M$ is zero, and hence
  $\alpha_1=\dots=\alpha_m=0$ and $Y_{(p,\ov p)}=0$. Now, since 
  the symmetry group of the absolute parallelism~$\E$ acts on~$M$ without fixed
  points, we have $Y=0$. It follows, in particular, that the intersection of the
  subspaces 
  $$\sym(\E)(p,\ov p)=\{\,Y_{(p,\ov p)}\in T_{(p,\ov p)}L\mid
  Y\in\sym(\E)\,\}$$ and $E_{(p,\ov p)}$ is equal to zero. Finally, $\dim E=m$,
  $\dim L=2m-r$, and $\dim\sym(\E)=m-r$, which implies that the Lie algebra
  $\sym(\E)|_L$ is indeed a simply transitive symmetry algebra of the
  distribution~$E$.
\end{proof}

\subsection{Cartan connections as absolute parallelisms}
\label{car:abs}

The application of the equivalence theory of absolute parallelisms
to the finding of invariants of Cartan connections has certain peculiarities,
which are due mostly to the presence of an additional structure, that of the
action of the Lie group~$G$ on the principal fiber bundle $\pi\colon P\to M$.

Let $\om$ be a Cartan connection on~$P$, $\Omega$ the curvature form of~$\om$,
and $c\colon P\to \Lgg$ the structure function of~$\om$. If
$\{e_1,\dots,e_{n+m}\}$ is a basis of the Lie algebra~$\gb$ such that
$\{e_{m+1},\dots,e_{n+m}\}$ is a basis of the subalgebra~$\g$, then the
structure function~$c$ is uniquely determined by the functions $c_{ij}^k$
$(i,j=1,\dots,m,k=1,\dots,n+m)$:
$$c\colon (e_i+\g)\land(e_j+\g)\mapsto\sum_{k=1}^{n+m}c_{ij}^k e_k.$$
It is not difficult to see that, up to structure constants of~$\gb$, 
the functions $c_{ij}^k$ are precisely the nonzero structure
functions of the absolute parallelism on~$P$ corresponding to the
Cartan connection~$\om$. Hence we may assume without loss of generality that 
$c\colon P\to \Lgg$ is the structure mapping of the first order corresponding
to the absolute parallelism.

Observe that the function~$c$ takes its values in a vector space in which there
is defined a linear action of the group~$G$. Moreover, 
$$c(pq)= g^{-1}.c(p)\quad \text{for all}~ p\in P,g\in G.$$
We shall now show that the structure functions of higher orders, which are
derived from~$c$ with the help of covariant differentiation, can also be
written in this form.

Suppose $\rho\colon G\to GL(V)$ is a representation of the group~$G$ on an
arbitrary vector space~$V$. We say that a function $f\colon P\to V$ is
\emph{$G$-invariant} ({\em of type~$\rho$\/}) if
$$f(pq)= g^{-1}.f(p)\ \forall p\in P,~g\in G.$$

We define the covariant derivative of~$f$ as a function $f^{(1)}$ on~$P$ with
values in the vector space $\mathcal L(\gb,V)=\gb^*\otimes V$ such that
$$f^{(1)}(p)\colon X\mapsto X_p^*f$$
for all $X\in\gb,p\in P$. Hence the functions $f_i\colon P\to V$,
$p\mapsto f^{(1)}(p)(e_i)$, $i=1,\dots,n+m$, are none other than the covariant
derivatives of~$f$ along the fundamental vector fields $e_1^*,\dots,e_{n+m}^*$.

Consider the following action of~$G$ on $\mathcal L(\gb,V)$:
$$(g.\v)(X)=g.\v(\Ad g^{-1}(X))\quad\text{for all } g\in G,\v\in\mathcal 
L(\gb,V), X\in\gb.$$

\begin{prop}\ 

1. The function $f^{(1)}$ is $G$-invariant.

2. For any $X\in \g$,
$$f^{(1)}(P)\colon X\mapsto -X.f,$$ 
where for $v\in V$, $X.v$ denotes the action of~$\g$ on~$V$ corresponding
to the representation~$\rho$.
\end{prop}

\begin{proof}\ 

1. Indeed, for any $p\in P, g\in G, X\in\gb$ we have
\begin{multline*}
f^{(1)}(pq)(X)=X_{pq}^*f=(\d_pR_q)(\Ad_g(X)_p^*)f=\\
=\Ad_g(X)_p^*(f\circ R_{g^{-1}})=\Ad_g(X)_p^*(g^{-1}.f)=\\
=g^{-1}.(\Ad_g(X)_p^*f)=(g^{-1}.f^{(1)}(p))(x).
\end{multline*}

2. This follows immediately from the definition of derivative along a vector
field and that of fundamental vector fields.
\end{proof}

Let $s\colon M\to P$ be a local section of the fiber bundle $\pi\colon P\to M$,
$\tilde\om=s^*\om$ a form on~$M$ with values in~$\gb$, and 
$\tilde f=f\circ s$ a function on~$M$ with values in the vector space~$V$.
Since $f$ is $G$-invariant, in view of the equality 
$$f(s(x)g)= g^{-1}.\tilde f(x)\quad\text{for all } x\in M,g\in G,$$
$f$ is determined uniquely by~$\tilde f$.
Show that to find the function $\tilde f^{(1)}=f^{(1)}\circ s$, we need only
know the function~$\tilde f$ and the form $\tilde\om$.

The function~$f^{(1)}$ is uniquely determined by the $V$-valued functions
$f_1,\dots,f_{n+m}$ in the decomposition
$$\d f=f_1\om^1+\dots+f_{n+m}\om^{n+m}.$$
Moreover, each vector~$e_i$ with $i\ge m+1$ lies in the subalgebra~$\g$, and
therefore $f_i=e_if=-e_i.f$. Furthermore, 
$$\d \tilde f=\tilde f_1\om^1+\dots+\tilde f_{n+m}\om^{n+m},$$
where $\tilde f_i=f\circ s$ for all $i=1,\dots,n+m$. It is clear that for $i\ge
m+1$, we have $\tilde f_i=-e_i.\tilde f$, and the functions $\tilde
f_1,\dots,\tilde f_m$ are uniquely determined by the equality
$$\tilde f_1\tilde\om^1+\dots+\tilde f_m\tilde\om^m=
\d\tilde f+(e_{m+1}.\tilde f)\om^{m+1}+\dots+(e_{n+m}.\tilde f)\om^{n+m}.$$

From the definition of Cartan connection it follows that the forms 
$\tom^1,\dots,\tom^m$ form a coframe on~$M$. Let $\TX_1,\dots,\TX_m$ be the
dual frame. Then
$$\tilde f_k=(\d \tilde f+\sum_{i=1}^n(e_{m+i}.\tilde f)\tom^{m+i})(\TX_k)=\\
\TX_i\tilde f+\sum_{i=1}^n\tom^{m+i}(\TX_k)(e_{m+i}.\tilde f).$$

Unlike covariant differentiation along~$X_k$ on~$P$, the right-hand side of
this expression is a nonhomogeneous linear differential operator of the first
order. 

Set $c_1=c$ and $c_{n+1}=c_n^{(1)}$ for $n\ge 1$. Then the structure
mapping~$C^{(s)}$ of the absolute parallelism defined by the
Cartan connection~$\omega$ may be written in the form
$$C^{(s)}=c_1+\dots +c_s.$$
In this notation, it takes values in the vector space
$$V^{(s)}=V\oplus \gb^*\otimes
V\oplus\dots\oplus\underbrace{\gb^*\otimes\dots\otimes \gb^*}_{s-1}\otimes V.$$

We remark that $V^{(s)}$ is endowed with a natural linear action of the Lie
group~$G$, and the mapping $C^{(s)}$ is $G$-invariant.

As an example, we shall prove the following theorem:
\begin{thr}[S.~Lie~\cite{lie2}]\ 

1. The symmetry group of any second-order equation is a Lie group of 
dimension~$\le8$. 

2. A second-order equation is locally equivalent to the equation
$y''=0$ if and only if one of the following conditions is satisfied:
\begin{itemize}
\item[(i)] the symmetry group of the equation in question has dimension~8;
\item[(ii)] the structure function of the corresponding Cartan connection is
  identically equal to zero.
\end{itemize}
\end{thr}
\begin{proof}
Any second-order equation defines a Cartan connection~$\om$ on the principal
fiber bundle $\pi\colon P\to J^1(\R^2)$, and the symmetries of this equation
may be uniquely extended to symmetries of~$\om$. On the other hand, $\om$ can
be regarded as an absolute parallelism on the 8-dimensional
manifold~$P$. The truth of the first assertion of our theorem now follows
immediately.

The symmetry group of the Cartan connection~$\om$ is 8-dimensional if and only
if the corresponding absolute parallelism has rank zero, which is
possible only when the structure function~$c$ of~$\om$ is constant.
But since
$$ c(pg)=g^{-1}.c(p), \quad\text{for all }p\in P,g\in G,$$
it immediately follows that the element $c\equiv c(p)\in W$ must be invariant
under the action of~$G$ on~$W$. The formulas~(\ref{act}) imply however that the
only invariant element of the space~$W$ is the zero vector. This proves the
equivalence of conditions (i) and (ii).

The Cartan connection corresponding to the equation $y''=0$ is precisely the
canonical Maurer--Cartan connection on the Lie group $\G=SL(3,\R)$, and its
curvature form is obviously equal to zero. Thus our original equation is
equivalent to the equation $y''=0$ if and only if the curvature form (and hence
the structure function of~$\om$) is zero.
\end{proof}

\section{Classification of second-order equations}

Consider a second-order equation $y''=f(x,y,y')$ and the
corresponding pair $(V,E)$ of direction fields in $J^1(\R^2)$.

\begin{defi}
A local diffeomorphism $\v\colon \R^2\to \R^2$ is called a \emph{symmetry} of
the equation $y''=f(x,y,y')$ if its first prolongation   
$\v^{(1)}\colon  J^1(\R^2)\to J^1(\R^2)$ preserves the pair $(V,E)$.
\end{defi}

A vector field on the plane is said to be an {\it infinitesimal symmetry\/}
of a second-order equation if the one-parameter group of local diffeomorphisms
generated by this field consists of symmetries of the equation under
consideration.

The set of all symmetries of a given second-order equation forms a Lie algebra.
If two equations are equivalent by means of a certain local diffeomorphism of
the plane, then it is clear that the symmetry algebras of these equations are 
equivalent with respect to the same local diffeomorphism.

The classification of second-order differential equations according to
their symmetry algebras was carried out (in the complex-analytical case) by
Lie~\cite{lie1} and improved by Tresse~\cite{tr2} (see also~\cite{olver2}). 
All the above constructions
are valid in the complex-analytical case, and in the following
we shall consider precisely this case. The real case requires only slight
modifications, which are omitted here. 

\begin{thr}[\cite{tr2}]\label{t32:1}
Let $y''=f(x,y,y')$ be a second-order equation, and $\h$ its symmetry algebra.

1. If $\dim \h=1$, then this equation is equivalent to an equation of the form
$y''=h(x,y')$ with symmetry algebra $\tilde\h=\langle \frac{\p}{\p y}\rangle.$

2a. If $\dim \h=2$ and the Lie algebra $\h$ is commutative, then the given
equation is equivalent to an equation of the form $y''=h(y')$ with symmetry
algebra $\tilde \h=\langle \frac{\p}{\p x}, \frac{\p}{\p y} \rangle.$

2b. If $\dim \h=2$ and $\h$ is not commutative, then the given equation is 
equivalent to an equation of the form $y''=\frac{h(y')}x$ with symmetry algebra
$\tilde \h=\langle \frac{\p}{\p y}, x\frac{\p}{\p x}+
y\frac{\p}{\p y}\rangle.$

3. If $\dim \h=3$, then the given equation is equivalent one of the following
equations: 

\begin{enumerate}[a)]
\item $y''=(y')^\al$, $\tilde \h=\langle \frac{\p}{\p x}, \frac{\p}{\p y},
x\frac{\p}{\p x}+cy\frac{\p}{\p y}\rangle$, $\al=\frac{c-2}{c-1}$\\
($\al\ne 0,1,2,3$; equations corresponding to $\al$ and $3-\al$ are equivalent);
\item $y''=(1+(y')^2)^{3/2}e^{-\al\arctg y'}$,
$\tilde\h = \langle \frac{\p}{\p x}, \frac{\p}{\p y}, (-y+\al x)\frac{\p}{\p x}+(x+\al y)\frac{\p}{\p y}\rangle$\\
($\al\ne0$; equations corresponding to parameters $\al$ and $-\al$
are equivalent);
\item $y''=e^{-y'}$, $\tilde \h=\langle \frac{\p}{\p x}, \frac{\p}{\p y},
x\frac{\p}{\p x}+(x+y)\frac{\p}{\p y}\rangle$;
\item $y''=\frac{\pm(y')^3-y'}{2x}$, $\tilde \h=\langle \frac{\p}{\p y},
x\frac{\p}{\p x}+y\frac{\p}{\p y}, 2xy\frac{\p}{\p x}+
y^2\frac{\p}{\p y}\rangle$;
\item\label{eq:e}
 $y''=\frac{y'(1-(y')^2)+\al|(y')^2-1|^{3/2}}x$, $\tilde
\h=\langle\frac{\p}{\p y}, x\frac{\p}{\p x}+y\frac{\p}{\p y},
2xy\frac{\p}{\p x}+(x^2+y^2)\frac{\p}{\p y}\rangle$\\
($\al\ne0$; equations corresponding to $\al$ and $-\al$ are equivalent);
\item $y''=\frac{y'(1+(y')^2)+\al(1+(y')^2)^{3/2}}x$, $\tilde
\h=\langle\frac{\p}{\p y}, x\frac{\p}{\p x}+y\frac{\p}{\p y},
2xy\frac{\p}{\p x}+(y^2-x^2)\frac{\p}{\p y}\rangle$\\
($\al\ne0$; equations corresponding to $\al$ and $-\al$ are equivalent);
\item $y''=\frac{2(1+(y')^2)(xy'-y)+\al(1+(y')^2)^{3/2}}{1+x^2+y^2}$,\\
$\tilde \h=\langle -y\frac{\p}{\p x}+x\frac{\p}{\p y}, (1+x^2-y^2)\frac{\p}{\p x}+2xy\frac{\p}{\p y},
2xy\frac{\p}{\p x}+(1-x^2+y^2)\frac{\p}{\p y}\rangle$\\
($\al\ne0$; equations corresponding to $\al$ and $-\al$ are equivalent);
\end{enumerate}

4. If $\dim \h>3$, then $\dim\h=8$ and the given equation is equivalent to the
equation $y''=0$.
\end{thr}

\begin{rem}
  In the case 1, he change of variables $z=y'$ reduces the equation
  $y''=h(x,y')$ to the first-order equation $z'=h(z,x)$. The equation~2a has
  the following general solution:
$$y=\int z(x)\d x, \text{ where } \int \frac{\d z}{h(z)}=x.$$
In the case 2b, the general solution has the form
$$y=\int z(x)\d x, \text{ where } \int \frac{\d z}{h(z)}=\ln x.$$

In case~3e all solutions of the given second order ODE satisfy one of the following conditions:
$(y')^2>1$, $(y')^2=1$ or $(y')^2<1$. This means that the equation viewed as a hypersurface 
$\E$ in $J^2(\R^2)$ can be represented as a union of three parts $\E_+$, $\E_0$ è $\E_-$ respectively. 
The sets $\E_+$ and $\E_-$ are open in $\E$ and can be viewed as two separate euqations. As we shall see 
they are not equivalent to each other. In other words, item~3\ref{eq:e} includes two different 
ODEs. We shall refer to them as 3e$_+$ and 3e$_-$ respectively.

Similarly, equations from~3d, corresponding to two different signs are not equivalent to each other and will 
be deonoted by 3d$_+$ and 3d$_-$ respectively.
\end{rem}

In connection with Theorem~\ref{t32:1} the following problems suggest
themselves: 

(A) Given a second-order equation, find the dimension of its symmetry
algebra~$\h$.

(B) If $\h\ne\{0\}$, determine which of the above equations is equivalent to
the given one.

(C) Find a local diffeomorphism establishing the equivalence.

Note that if we are able to solve the problems (A)--(C), then in the case $\dim
\h=1$ this enables us to reduce the given equation to a first-order equation,
while in the case $\dim \h\ge 2$ the general solution may be found explicitly.

We shall consider these problems in terms of the canonical Cartan
connection~$\omega$ corresponding to the second-order equation in question.

\smallskip
{\bf(A)} Symmetries of the pair $(V,E)$ of direction fields are in one-to-one
correspondence with symmetries of the Cartan connection~$\omega$,
regarded as an absolute parallelism. If $r$ is the rank of this
absolute parallelism, then $\dim\h=8-r$.

\smallskip
{\bf(B)} If $\dim\h=1$ or $\dim\h>3$, then our equation is equivalent to a
uniquely determined equation from Theorem~\ref{t32:1}. In the case $\dim\h=2$ we
make use of the fact that the structure of the symmetry algebra of an
absolute parallelism coincides with the structure of the Lie algebra
obtained by restricting the frame to an arbitrary orbit of the symmetry group.
Hence, if this restriction is a commutative Lie algebra, then our initial
equation is equivalent to the equation $y''=h(y')$; otherwise, it may be
transformed into the equation $y''=\frac{h(y')}x$.

These considerations, however, cannot be applied to the case
$\dim\h=3$, since, for example, the symmetry algebras of the equations~3c, 
~3d and 3f in Theorem~\ref{t32:1} are isomorphic to $\mathfrak{sl}(2,\R).$

Consider the function
$$f\colon P\to \R^2, p\mapsto\begin{pmatrix} a(p)\\ d(p)\end{pmatrix},$$
where the functions $a,d$ are uniquely defined by 
$$\Omega_{12}=a\omega_{21}\land\omega_{31},$$
$$\Omega_{23}=d\omega_{31}\land\omega_{32}.$$
It follows from~(\ref{act}) that
$$f(pq)=g^{-1}.f(p),$$
where $g=\begin{pmatrix} x&t&v\\ 0&y&u\\ 0&0&z \end{pmatrix}$
acts on $\R^2$ in the following way:
\begin{equation}\label{e32:1}
  g.\begin{pmatrix} a\\ d\end{pmatrix}=\begin{pmatrix} \frac{x^3}{y^2z}&0\\ 
  0&\frac{xy^2}{z^3}\end{pmatrix} \begin{pmatrix} a\\ d\end{pmatrix}.
\end{equation}
The function~$f$ is therefore $G$-invariant.

We fix the following basis in the Lie algebra $\gb=\mathfrak{sl}(3,\R)$:

$$u_1=\begin{pmatrix} 0&0&0\\ 1&0&0\\ 0&0&0 \end{pmatrix},\
u_2=\begin{pmatrix} 0&0&0\\ 0&0&0\\ 0&1&0 \end{pmatrix},\ 
u_3=\begin{pmatrix} 0&0&0\\ 0&0&0\\ 1&0&0 \end{pmatrix},$$
$$e_1=\begin{pmatrix} -1/3&0&0\\ 0&2/3&0\\ 0&0&-1/3 \end{pmatrix},\ 
e_2=\begin{pmatrix} -1/3&0&0\\ 0&-1/3&0\\ 0&0&2/3 \end{pmatrix},$$
$$e_3=\begin{pmatrix} 0&1&0\\ 0&0&0\\ 0&0&0 \end{pmatrix},\ 
e_4=\begin{pmatrix} 0&0&0\\ 0&0&1\\ 0&0&0 \end{pmatrix},\ 
e_5=\begin{pmatrix} 0&0&1\\ 0&0&0\\ 0&0&0 \end{pmatrix},$$
so that the vectors $e_1,\dots,e_5$ form a basis of the subalgebra~$\g$.
The vector fields $u_i^*$, $i=1,\dots,3,$ $e_j^*$, $j=1,\dots,5$,
on~$P$ define an absolute parallelism, and the dual coframe has
the form:
$$\{\omega_{21},\omega_{32},\omega_{31},\omega_{22}-\omega_{11},
\omega_{33}-\omega_{11},\omega_{12},\omega_{23},\omega_{13}\}.$$

\begin{lem} The functions~$b$ and~$c$ on~$P$ uniquely defined by 
$$\Omega_{13}=b\omega_{21}\land\omega_{31}+c\omega_{31}\land\omega_{32}$$
are the covariant derivatives of the functions $a$ and $(-d)$ along the vector
fields $u_2^*$ and $u_1^*$ respectively:
$$b=u_2^*a, c=-u_1^*d.$$
\end{lem}

\begin{proof}
This easily follows from the structure equation $\d\Omega=[\Omega,\omega].$
\end{proof}

Let $s\colon \R^3\to P$ be the same equation as the one used in the
construction of the Cartan connection~$\omega$.

Then the forms $\tom_{21}=\d x$, $\tom_{32}=\d z-f\d x$, and
$\tom_{31}=\d y-z\d x$, form the coframe on~$\C^3$ which is the dual of the
frame 
$$\TX_1=\frac{\d}{\d x},~\TX_2=\frac{\p}{\p z},~\TX_3=\frac{\p}{\p y}.$$
We let $\tilde f=s^*f=\begin{pmatrix} \tilde a\\ \tilde d\end{pmatrix}$ and find the
function 
$$\tilde f^{(1)}\colon \R^3\to\mathcal L(\gb,\R^2).$$
The action of the Lie algebra~$\g$ on~$\R^2$ corresponding to the
action~(\ref{e32:1}) of the Lie group~$G$ has the form
$$e_1.\begin{pmatrix} \tilde a\\ \tilde d\end{pmatrix}=\begin{pmatrix} -2\tilde a\\ 2\tilde
d\end{pmatrix}, 
\ e_2.\begin{pmatrix} \tilde a\\ \tilde d\end{pmatrix}=\begin{pmatrix} -\tilde a\\ -3\tilde
d\end{pmatrix},$$ 
$$e_i.\begin{pmatrix} \tilde a\\ \tilde d\end{pmatrix}=0,\ i=3,4,5.$$

If now $\tilde f_i=(u_i^*f)\circ s,$ $i=1,2,3$, then we have
$$\tilde f_1\tom_{21}+\tilde f_2\tom_{32}+\tilde f_3\tom_{31}=\\
\d \tilde f+\begin{pmatrix} -2\tilde a\\ 2\tilde d\end{pmatrix}(\tom_{22}-\tom_{11})+
\begin{pmatrix} -\tilde a\\ -3\tilde d\end{pmatrix}(\tom_{33}-\tom_{11}),$$
and it follows immediately that
\begin{multline*}
\tilde f_1=\TX_1\begin{pmatrix} \tilde a\\ \tilde d\end{pmatrix}+(\tom_{22}-\tom_{11})
(\TX_1)\begin{pmatrix} -2\tilde a\\ 2\tilde d\end{pmatrix}+\\
+(\tom_{33}-\tom_{11})(\TX_1)\begin{pmatrix} -\tilde a\\ -3\tilde d\end{pmatrix}=
\begin{pmatrix} \frac{\d\tilde a}{\d x}-2f_z\tilde a\\[1mm] \frac{\d\tilde d}{\d x}+
2f_z\tilde d\end{pmatrix}.
\end{multline*}
Similarly,
$$\tilde f_2=\begin{pmatrix} 
\frac{\p\tilde a}{\p z}\\[1mm] \frac{\p\tilde d}{\p z}\end{pmatrix},\ 
\tilde f_3=\begin{pmatrix} \frac{\p\tilde a}{\p y}+1/2f_{zz}\tilde a\\[1mm]
\frac{\p\tilde d}{\p y}+3/2f_{zz}\tilde d\end{pmatrix},$$
and, in particular,
$$\tilde b=\frac{\p\tilde a}{\p z}, \tilde c=-\frac{\d 
\tilde d}{\d x}-2f_z\tilde d.$$

Thus, identifying the elements of the space ${\mathcal L}(\gb,\R^2)$ with 
$2\times 8$-matrices, we obtain
$$\tilde f^{(1)}=\begin{pmatrix} c_{11}&c_{12}&c_{13}&2\tilde a&\tilde a&0&0&0\\
 c_{21}&c_{22}&c_{23}&-2\tilde d&3\tilde d&0&0&0\end{pmatrix},$$
where $\begin{pmatrix} c_{1i}\\ c_{2i}\end{pmatrix}=\tilde f_i$ for $i=1,2,3.$

Since the elements of the image of~$\tilde f^{(1)}$ are zero on the subalgebra
$\g_1=\langle e_3,e_4,e_5\rangle$, we can assume that $\tilde f^{(1)}$
takes values in the space $\mathcal L(\gb/ {\g_1},\R^2)$.
Then the elements $u_i+\g_1$, $i=1,2,3$, $e_j+\g_1$, $j=1,2$, form a basis of
the space $\gb/ {\g_1}$, and in this basis,
\begin{equation}\label{e32:2}
\tilde f^{(1)}=\begin{pmatrix} c_{11}&c_{12}&c_{13}&2\tilde a&\tilde a\\
 c_{21}&c_{22}&c_{23}&-2\tilde d&3\tilde d\end{pmatrix}.
\end{equation}

An element $g=\begin{pmatrix} x&t&v\\ 0&y&u\\ 0&0&z \end{pmatrix}$
acts on the space $\mathcal L(\gb/ \g_1,\C^2)$ according to the formula
$g.A=X_1 A X_2^{-1}$, where
$$ X_1= \begin{pmatrix} \frac{x^3}{y^2z}&0\\ 0&\frac{xy^2}{z^3}\end{pmatrix},
X_2=\begin{pmatrix} \frac yx&0&\frac ux&0&0\\
0&\frac zy&-\frac{zt}{xy}&0&0\\ 0&0&\frac zx&0&0\\
-\frac{2t}x&\frac uy &-\frac{yv+ut}{xy}&1&0\\-\frac tx&-\frac uy&\frac{-2yv+ut}{xy}&0&1
\end{pmatrix}.$$

Let us divite euqations with 3-dimensional symmetry alegrba from Theorem~\ref{e32:1} 
into three families depending on whether the values of invariants $\tilde a$ and $\tilde d$
vanish or not. Direct calculation shows that:
\begin{itemize}
\item[] $\tilde a=0,\,\tilde d\ne0$: 3e$_+$\,($\alpha=\pm1$),
3f\,($\alpha=\pm1$);
\item[] $\tilde a\ne0,\,\tilde d=0$: 3d$_{\pm}$;
\item[] $\tilde a\ne0,\,\tilde d\ne0$: all other equations.
\end{itemize}

Consider first the case $\tilde a\ne 0, \tilde d\ne 0$. Then the matrix~(\ref{e32:2})
can always be brought to the form
\begin{equation}\label{e32:3}
\begin{pmatrix} s_{1}&s_{2}&s_{3}&2\tilde a&\tilde a\\
0&0&0&-2\tilde d&3\tilde d\end{pmatrix}.\end{equation}
This may be done by choosing the element $g\in G$ with
$$x=y=z=1,\ t=\frac{c_{21}}{\tilde d},\ u=-\frac{c_{22}}{5\tilde d},
\ v=\frac{-ut\tilde d+uc_{21}-tc_{22}-c_{23}}{4\tilde d}.$$

Then the functions $s_i$, $i=1,2,3$, $\tilde a$, $\tilde d$ are determined
uniquely up to transformations of the form
\begin{equation}\label{eq:si1}
\begin{pmatrix}
\tilde a\\ \tilde d
\end{pmatrix}\mapsto
\begin{pmatrix} \frac{x^3}{y^2z}\tilde a
\\[1mm] \frac{xy^2}{z^3}\tilde d
\end{pmatrix},
\end{equation}

\begin{equation}\label{eq:si2}
\begin{pmatrix} s_1\\s_2\\s_3
\end{pmatrix}\mapsto
\begin{pmatrix} \frac{x^4}{y^3z}s_1
\\[1mm] \frac{x^3}{yz^2}s_2\\[1mm] \frac{x^4}{y^2z^2}s_3
\end{pmatrix}.
\end{equation}

Thus the functions $\tilde a,\tilde d$, $s_1,s_2,s_3$ are
\emph{semi-invariants} of the original equation $y''=f(x,y,y')$. A
straightforward computation of these semi-invariants for the equation 3a--3c, 3e--3g
from Theorem~\ref{t32:1} shows that this set of semi-invariants is enough to
identify the equation equivalent to a given equation with three-dimensional
symmetry algebra. The invariant conditions below give various classes of
equations equivalent to the equations 3a--3c, 3e--3g from Theorem~\ref{t32:1}.

1. $s_1,s_2\ne0$. This case includes equations 3a ($\alpha\ne3/2$),
3b ($\alpha\ne0$) and 3c. Then we can define the invarant 
$I_1=\frac{\tilde as_3}{s_1s_2}$,  which takes the following values on equations 3a--3c:

\begin{center}
\begin{tabular}{|c|c|c|}\hline
Equations & Parameter & Invariant $I_1$ \\ \hline
3a & $\alpha\ne3/2$ &
$\frac{41\alpha(\alpha-3)+96}{256\alpha(\alpha-3)+576}>\frac{41}{256}$ \\ \hline
3b & $\alpha\ne0$ &
$\frac{41\alpha^2-15}{256\alpha^2}<\frac{41}{256}$ \\ \hline
3c & & $\frac{41}{256}$ \\ \hline
\end{tabular}
\end{center}

Thus, the value of $I_1$ is different for all non-equivalent equations from items 3a--3c 
and allows to identify uniquely the corresponding equation.

2. $s_1=s_2=0$. This case includes 3a ($\alpha=3/2$),
3b ($\alpha=0$), 3e (except 3e$_+$ for $\alpha=\pm1$), 3f ($\alpha\ne\pm1$),
3g. It turns out that in all these cases the semi-invariant $s_3$ does not vanish. So, we can
define the invariant $I_2=\frac{\tilde a^5\tilde d}{s_3^4}$. Moreover, 
equations~(\ref{eq:si1}) and~(\ref{eq:si2}) imply that the signs of semi-invariants 
$\tilde a\tilde d$ and $s_3$ have invariant meaning. The table below lists the values of these invariants 
for the euqations from this case.

\begin{center}
\begin{tabular}{|c|c|>{\hskip10mm}m{27mm}|>{\hskip10mm}m{27mm}|c|}\hline
\raisebox{-1.5ex}[0pt][0pt]{Equation} & Condtions on &
\multicolumn{2}{|c|}{Signs of semi-invariants} & Value of the \\ \cline{3-4}
& parameters & $\tilde a\tilde d$ & $s_3$ & invariant $I_2$ \\ \hline
3a     & $\alpha=3/2$ & $+$ & $-$ & $\frac1{36}$ \\ \hline
3b     & $\alpha=0$ & $+$ & $+$ & $\frac1{36}$ \\ \hline
3e$_+$ & $0<|\alpha|<1$ & $-$ & $+$ &
$\frac1{36}\frac{\alpha^2-1}{\alpha^2}<\frac1{36}$ \\
       & $|\alpha|>1$ & $+$ & $-$ &
$\frac1{36}\frac{\alpha^2-1}{\alpha^2}<\frac1{36}$ \\ \hline
3e$_-$ & $\alpha\ne0$  & $+$ & $-$ &
$\frac1{36}\frac{\alpha^2+1}{\alpha^2}>\frac1{36}$ \\ \hline
3f & $0<|\alpha|<1$ & $-$ & $-$ &
$\frac1{36}\frac{\alpha^2-1}{\alpha^2}<\frac1{36}$ \\
   & $|\alpha|>1$ & $+$ & $+$ &
$\frac1{36}\frac{\alpha^2-1}{\alpha^2}<\frac1{36}$ \\ \hline
3g & $\alpha\ne0$  & $+$ & $+$ &
$\frac1{36}\frac{\alpha^2+1}{\alpha^2}>\frac1{36}$ \\ \hline
\end{tabular}
\end{center}
Again, this table implies that the given invariants allow to recognize the corresponding 
equation in a unique way.

Consider now the case $\tilde d=0,\,\tilde a\ne0$. Then the second row
of the matrix~(\ref{e32:2}) vanishes identically, and function
$\tilde f^{(1)}$ is not sufficient to distinguish between equations 3d$_+$ and
3d$_-$. Similar to the fuction $f$, we can consider the fuction
$$h\colon P\to \R^4,\quad p\mapsto
\begin{pmatrix} a(p) \\ a_1(p) \\ a_2(p) \\ a_3(p)
\end{pmatrix},$$
where $a_i=u_i^*a$ for $i=1,\dots,3$. Then this function is equivariant: $h(pg)=g^{-1}.h(p)$, 
where the action of $g\in G$ on $\R^4$ is given by the following matrix:
$$
\rho(g)=\begin{pmatrix} \frac{x^3}{y^2z} & 0 & 0 & 0 \\
\frac{5x^3t}{y^3z} & \frac{x^4}{y^3z} & 0 & 0 \\
-\frac{x^3u}{y^2z^2} & 0 & \frac{x^3}{yz^2} & 0 \\
\frac{x^3(4yv-5ut)}{y^3z^2} & -\frac{x^4u}{y^3z^2} &
\frac{x^3t}{y^2z^2} & \frac{x^4}{y^2z^2}\end{pmatrix}.
$$
Fix a section $s\colon\R^3\to P$ and consider $\tilde h=s^*h$.
We get
$$\tilde h=\begin{pmatrix} \tilde a \\ \tilde a_1 \\ \tilde a_2 \\ \tilde a_3
\end{pmatrix},\quad \text{where } \quad
\begin{aligned}
\tilde a_1&=\frac{d\tilde a}{dx}-2f_z\tilde a,\\
\tilde a_2&=\frac{\p\tilde a}{\p z},\\
\tilde a_3&=\frac{\p\tilde a}{\p y}+\frac12f_{zz}\tilde a.
\end{aligned}
$$

Consdier the function $h^{(1)}\colon P\to \L(\gb,\R^4)$.
Then, using the results of subsection~\ref{car:abs}, we get:
$$
h^{(1)}=
\begin{pmatrix}
a_1 & a_2 & a_3 & 2a & a & 0 & 0 & 0 \\
a_{11} & a_{12} & a_{13} & 3a_1 & a_1 & -5a & 0 & 0 \\
a_{21} & a_{22} & a_{23} & a_2 & 2a_2 & 0 & a & 0 \\
a_{31} & a_{32} & a_{33} & 2a_3 & 2a_3 & -a_2 & a_1 & -4a
\end{pmatrix},
$$
where $a_{ij}=u_i^*u_j^*a$ for all $i,j=1,\dots,3$. Then the structure
equation $d\Om=[\Om,\om]$ implies that $a_{22}=-(u_1^*)^2\,d=0$. In addtion, using
the Lie bracket relations among vector fields $u_1^*$, $u_2^*$, $u_3^*$ we get:
\begin{equation*}
\begin{aligned}
a_{21}&=a_{12}-a_3,\\
a_{31}&=a_{13},\\
a_{32}&=a_{23}.
\end{aligned}
\end{equation*}

Elements $g\in G$ act on $\L(\gb,\R^4)$ as follows:
$g.\phi=\rho(g)\phi\Ad g^{-1}$. Using this formula, it is easy to show that there is section
$s\colon \R^3\to P$, such that $\tilde h^{(1)}=s^*h^{(1)}$ has the form:
$$
\tilde h^{(1)}=
\begin{pmatrix}
0 & 0 & 0 & 2\tilde a & \tilde a & 0 & 0 & 0 \\
s_{11} & s_{12} & s_{13} & 0 & 0 & -5\tilde a & 0 & 0 \\
s_{12} & 0 & s_{23} & 0 & 0 & 0 & \tilde a & 0 \\
s_{13} & s_{23} & s_{33} & 0 & 0 & 0 & 0 & -4\tilde a
\end{pmatrix}.
$$
Here functions $s_{11}$, $s_{12}$, $s_{13}$, $s_{23}$ and $s_{33}$ are semi-invariants
defined up to the following scaling:
$$(s_{11},s_{12},s_{13},s_{23},s_{33})\mapsto
\left(\frac{x^5}{y^4z}s_{11}, \frac{x^4}{y^2z^2}s_{12}, \frac{x^5}{y^3z^2}s_{13},
\frac{x^4}{yz^3}s_{23}, \frac{x^5}{y^2z^3}s_{33}\right).$$

\enlargethispage*{1cm}
Explicit computation of these semi-invariants for equation~3d shoaw that
$s_{13}=s_{23}=0$ in this case, and the three remianing invariants do not vanish.
These three semi-invariants allow us to form two invariants:
$I_1=\frac{s_{11}s_{12}}{\tilde a^3}$ and
$I_2=\frac{s_{12}^2}{\tilde a s_{33}}$. They take values $25/12$ and $-5/4$ respectively foê~3d. 
Finally, the sign of the semi-invariant $s_{12}$ is also preserved. It is positive for 3d$_+$ and negative 
in case of 3d$_-$.

The case $\tilde a=0,\,\tilde d\ne0$ can be considered in a similar way. We just note that euqations
3e$_+$ ($\alpha=\pm1$) and 3f ($\alpha=\pm1$) are dual to 3d$_+$ and 3d$_-$ respectively.

\smallskip
{\bf (C)} With the help of the concept of Cartan connection, the problem of
finding a diffeomorphism transforming one second-order equation into another
may be reduced to a similar problem for absolute parallelisms.

Note that we know the symmetry algebras for the equations listed in Theorem~7,
so that we can give the explicit form of the symmetry algebras of the
corresponding absolute parallelisms. Then, as follows from Lemma~5.2, the
problem of bringing a given equation to the canonical form reduces to the
integration of some completely integrable distribution with simply
transitive symmetry algebra isomorphic to the symmetry algebra of the original
equation.

The last problem is studied in Appendix~A. In particular, Theorem~12 shows that
the explicit form of the desired transformation may be found in quadratures if
the symmetry algebra is solvable. In Theorem~7, the only equations that have
solvable symmetry algebras are equations 1, 2a, 2b, 3a, 3b and~3ñ. In particular,
this allows to construct the general solutions of all equations equivalent to
~2a, 2b, 3a, 3b, 3c in quadratures.

\section*{Appendix A. Symmetries of completely integrable distributions}
\renewcommand{\thesection}{A}

\subsection{Basic definitions}

Let $M$ be a smooth manifold of dimension $n+m$. To each point $p\in M$ we
assign a subspace~$E_p$ of dimension~$m$ in the tangent space~$T_pM$.
Assume that $E_p$ depends smoothly on~$p$. Then the family $\{E_p\}$ is
called a {\em distribution\/} on~$M$.

Consider the sets
\begin{align*}
\D(E) &=\bigl\{\,X\in\Vect M\bigm|X_p\in E_p,~\forall p\in M\,\bigr\},\\
\L(E) &=\bigl\{\,\omega\in\L^1M\bigm|\omega(X)=0,~\forall\,X\in\D(E)\,\bigr\}.
\end{align*}
Both of the sets are modules over the ring $\C$ of smooth functions on~$M$.
The distribution~$E$ is uniquely determined by the module~$\D(E)$, or by the
module~$\L(E)$.

\begin{exmps}\ 
 
\textbf{1.}
Consider the differential equation
$$y^{(n)}=f\bigl(x,y,y',\dots,y^{(n-1)}\bigr).$$
Let
$$y=y(x),~y_1=y'(x),~\dots,~y_{n-1}=y^{(n-1)}(x),$$
$$M=\R^{n+1}.$$
The forms 
\begin{align*}
\omega_0&=dy-y_1\,dx,\\
  &\vdots\\
\omega_{n-2}&=dy_{n-2}-y_{n-1}\,dx,\\
\omega_{n-1}&=dy_{n-1}-f(x,p_0,\dots,p_{n-1})\,dx.
\end{align*}
define a distribution of dimension~1 on~$M$. 
The module~$\D(E)$ is generated by the field
$$\po x+y_1\po{y}+\dots+y_{n-1}\po{y_{n-2}}+f\po{y_{n-1}}.$$

\smallskip
\textbf{2.} {\bf (Contact distribution)} Let
$$M=\R^3,\quad\omega=dy-z\,dx.$$
The form~$\omega$ determines a distribution of dimension~2 on~$\R^3$. The
module $\D(E)$ is generated by the fields
$$X_1=\po x+z\,\po y,\quad X_2=\po z.$$
\end{exmps}

\begin{defi} Let $E$ be a distribution on a manifold~$M$.
A submanifold $L\subset M$ is called an {\em integral
manifold\/} of~$E$ if
$$T_pL\subset E_p$$
for all $p\in L$.
\end{defi}

The central object of the theory of distributions is finding integral
manifolds.

\begin{defi} We shall say that a distribution~$E$ on a manifold~$M$
is {\em completely integrable\/} if for any point $p\in M$, there exists an
integral manifold~$L$ of~$E$ such that $p\in L$ and such that
$$\dim L=\dim E,$$
that is $L$ has the greatest possible dimension.
\end{defi}

The distribution of the first example is completely integrable. This follows
from the fact that direction fields are locally rectifiable (the unique
existence theorem for ordinary differential equations).  We shall show later
that the Cartan distribution in the second example is not completely
integrable.

The following theorem gives a criterion to determine whether a distribution is
completely integrable.

\begin{thr}[Frobenius] 
The following conditions are equivalent:
\begin{enumerate}
\item $E$ is completely integrable.
\item The module $\D(E)$ is closed under the operation of commutation of
vector fields, that is $\D(E)$ is a Lie algebra.
\item The differential $d\omega$ of an arbitrary form $\omega\in\L(E)$ lies
in the ideal of~$\L^*M$ generated by the set~$\L(E)$, or in other words,
$$d\omega=\sum\gamma_j\wedge\omega_j$$
with $\omega_j\in\L(E)$, $\gamma_j\in\L^1(M)$.
\item In a certain neighborhood $U$ in~$M$, there exists a local
coordinate system
$(x_1,\dots,x_m,x_{m+1},\dots,x_{m+n})$ such that the intersection of~$U$
and an arbitrary maximal integral manifold of~$E$ has the form
$$x_{m+1}=\const,~\dots,~x_{m+n}=\const.$$
\end{enumerate}
\end{thr}

\begin{rem} If $E$ is given by $n$ differential 1-forms
$\omega_1,\dots,\omega_n$, then these forms generate the ideal
mentioned in the third condition
of the Theorem. Then the relations $d\omega_i=\sum\gamma_{ij}
\wedge\omega_j$ are equivalent to the relations
$$d\omega_i\wedge\omega_1\wedge\dots\wedge\omega_n=0,\quad i=1,\dots,n.$$
\end{rem}

\begin{ex}
Let us show that the Cartan distribution is not completely integrable. Indeed,
it is determined by the 1-form $\omega=dy-z\,dx$, and
$$d\omega\wedge\omega=dx\wedge dz\wedge\omega=-dx\wedge dy\wedge dz\ne0.$$
\end{ex}

A diffeomorphism $\phi\colon M\to M$ is called a {\em symmetry of a
distribution\/}~$E$ if
$$d\phi_p(E_p)=E_{\phi(p)}$$
for every point $p\in M$.

Thus symmetries preserve the distribution by mapping its subspaces into
each other. If $L$~is an integral manifold of the distribution~$E$, then
its image $\phi(L)$ by a symmetry~$\phi$ is also an integral manifold.

\begin{defi} 
  A vector field $X\in\Vect M$ is called an {\em infinitesimal symmetry\/} (or
  simply a {\em symmetry\/}) of a distribution~$E$ if the corresponding
  one-parameter local transformation group $\phi_t$ consists of symmetries.
\end{defi}

Denote by $\sym(E)$ the set of all infinitesimal symmetries of the
distribution~$E$. The next theorem allows to determine whether a vector
field $X$ is a symmetry of $E$.

\begin{thr} The following conditions are equivalent:
\begin{enumerate}
\item $X\in\sym(E)$;
\item $[X,~\D(E)]\subset\D(E)$.
\end{enumerate}
\end{thr}

\begin{defi} A function $f\in\C$ is called a {\em first integral\/}
of a distribution~$E$ if it is constant on each integral manifold of~$E$.
\end{defi}
We shall denote the set of all first integrals of~$E$ by~$I(E)$.
It is easy to see that the following three conditions are
equivalent:
\begin{enumerate}
\item $f\in I(E)$;
\item $Yf=0$ for all $Y\in\D(E)$;
\item $df\in\L(E)$.
\end{enumerate}

\subsection{Differential forms with values in a Lie algebra}
Suppose $\g$ is an arbitrary finite-dimensional Lie algebra and $M$ is a smooth
manifold. Denote by $\Lambda(M)$ the graded algebra (with respect to exterior
product) of differential $k$-forms on~$M$. 
Let $\Lambda(M,\g)=\Lambda(M)\otimes\g$. Then $\Lambda(M,\g)$ is an algebra
with respect to the operation 
$$ [\om_1\otimes x_1, \om_2\otimes x_2]=\om_1\wedge\om_2\otimes[x_1,x_2],\
\om_1,\om_2\in\Lambda(M),\,x_1,x_2\in\g.$$
For the sake of simplicity, when dealing with the elements of $\Lambda(M,\g)$,
we shall sometimes omit the symbol~$\otimes$.

The algebra $\Lambda(M,\g)$ inherits the graduation of $\Lambda(M)$: 
$$\Lambda^k(M,\g)=\Lambda^k(M)\otimes\g$$ 
and is therefore a graded algebra. Since the multiplication in a Lie algebra is
anticommutative, the multiplication in $\Lambda(M,\g)$ has the following property:
$$[\om_1,\om_2]=(-1)^{k_1k_2+1}[\om_1,\om_2],\quad\om_i\in\Lambda^k_i(M,\g).$$

If $e_1,\dots,e_n$ is a basis for the Lie algebra~$\g$, then every element
$\om\in\Lambda(M,\g)$ can be uniquely expressed in the form
$$\om=\om_1e_1+\dots+\om_ne_n,\quad \om_1,\dots,\om_n\in\Lambda(M).$$

The operation of exterior differentiation can also be extended to
$\Lambda(M,\g)$: 
$$d\colon\Lambda^k(M,\g)\to\Lambda^{k+1}(M,\g),\ \om\otimes 
x\mapsto d\om\otimes x.$$
Thus defined, it has the standard properties of exterior differentiation:
$$d\circ d=0,\quad d([\om_1,\om_2])=[d\om_1,\om_2]+(-1)^{k_1}[\om_1,d\om_2],\ 
\om_i\in\Lambda^{k_i}(M,\g).$$

Finally, notice that an element $\om\in\Lambda^k(M,\g)$ may be interpreted as a
family of skew-symmetric $k$-forms
$$\om_p\colon\Lambda^k(T_pM)\to\g,\quad p\in M$$
on the tangent spaces to $M$, with values in~$\g$.

\subsection{Transitive symmetry algebras}

Let $M$ --- be a smooth manifold of dimension $n+m$, and let $E$ be 
an $m$-dimensional completely integrable distribution on~$M$. 
Using the Frobenius theorem, it is easy to show that the 
distribution~$E$ is determined uniquely by its algebra $I(E)$ of first
integrals, namely
$$ E_p=\bigcup_{f\in I(E)} \ker d_pf \quad \text{for all }p\in M.$$

The set of all symmetries of~$E$ is precisely the 
normalizer of the subalgebra $\D(E)$ in the Lie algebra $\D(M)$, and hence it
is a subalgebra of $\D(M)$. We note that $\D(E)$ is, by
definition, an ideal in $\sym(E)$. The elements of $\D(E)$ are called the 
{\em characteristic symmetries\/} of~$E$.

It follows immediately from the definitions that 
\begin{alignat*}{2}
  &1.\quad& &\sym(E)(I(E))\subset I(E);\\ 
  &2.& &I(E)\cdot \sym(E)
  \subset \sym(E).
\end{alignat*}

Let $\g$ be a subalgebra $\g$ of $\sym(E)$.  If $p$ is a point in~$M$, then let
\begin{align*}
  \g(p)&=\{X_p \in T_pM\mid X\in \g\};\\ \g_p&= \{X\in \g\mid X_p\in E_p\}.
\end{align*} 
We point out that $\g(p)$ is a subspace of $T_p(M)$, whereas $\g_p$ is a
subspace (and indeed a subalgebra) of~$\g$.

Let $\aa=\g\cap\D(E)$. Since $\D(E)$ is an ideal in $\sym(E)$, we see that
$\aa$ is an ideal in~$\g$. Notice that $\aa$ can also be defined as
$\aa=\cap_{p\in M}\g_p$.

\begin{defi}
  A Lie algebra $\g\subset \sym(E)$ is called a {\em transitive symmetry
    algebra} of~$E$ if $\g(p)+E_p=T_pM$ for all $p\in M$. 

  A transitive symmetry algebra $\g\subset \sym(E)$ is called {\em simply
    transitive} if $\g_p=\aa$ for all $p\in M$.
\end{defi}
 
Let $\g\subset \sym(E)$ be a symmetry algebra of~$E$, and let 
$G$ be the local Lie transformation group generated by~$\g$.
Then $G$ preserves the set $\M$ of all maximal integral manifolds of~$E$, and
moreover the following assertions are true:
\begin{enumerate}
\item The ideal $\aa$ is zero if and only if the action of~$G$ on~$\M$ is
  locally effective.
\item The subalgebra $\g_p$ is precisely the Lie algebra of the subgroup
  $G_p=\{g\in G\mid g.L_p\subset L_p\}$;
\item The Lie algebra $\g$ is transitive if and only if the action of~$G$
  on~$M$ is locally transitive.             
\item The Lie algebra $\g$ is simply transitive if and only if 
  $G_p$~does not depend on~$p$ and hence coincides with the ineffectiveness
  kernel of the action of~$G$ on~$\M$.
\end{enumerate}

\subsection{Normalizer theorem}
\label{sec:n}

Let $\g$ be a transitive symmetry algebra of~$E$. We fix a point~$a$ in~$M$ and
consider the following set of points in~$M$:
$$\{p\in M\mid \g_p=\g_a\}.$$ 
Let $S_a$ denote the connected component of this set that contains~$a$. 

Since $\g$ is a transitive symmetry algebra, there exist $n$ vector fields 
$X_1,\dots,X_n\in\g$ such that the vectors
$$(X_1)_a,\dots,(X_n)_a\in T_aM$$ 
form a basis for the complement of $E_a$ in $T_aM$, and hence this will also be
true in some neighborhood~$U$ of~$a$:
$$ \langle (X_1)_p,\dots,(X_n)_p\rangle\oplus E_p=T_pM \text{ for all }p\in
U.$$

In the neighborhood~$U$, every vector field $Y\in\D(M)$ can be
written uniquely in the form 
\begin{equation}\label{1}
  Y=f_1X_1+\dots +f_nX_n \pmod{\D(E)}.
\end{equation}
Note that $Y\in\g$ belongs to $\g_p$ ($p\in U$) if and only
if 
$$ f_1(p)=\dots=f_n(p)=0.$$

Let $\F=\{f_\alpha\}$ be the family of all functions that appear in the
expansion~(\ref{1}) for all $Y\in\g_a$. Note that, for every $p\in M$, the
subalgebra $\g_p$ has the same codimension in~$\g$, which is equal to 
$\codim(E)$. Therefore, the equality $\g_p=\g_a$ is equivalent to the
inclusion $\g_p\subset\g_a$, which, for $p\in U$, can be written as
\begin{equation}\label{2}
  f(p)=0\quad \forall f\in\F.
\end{equation}
Thus, in a neighborhood of~$a$, the subset $S_a$ is given by the simultaneous
equations~(\ref{2}). Moreover, it is easy to show that $S_a$ is a 
submanifold in~$M$.

\begin{thr}\label{tn} 
Let $S=S_a$. Then
\begin{enumerate}
\item $T_pS\supset E_p$;
\item the subspace
  $$\{X\in\g\mid X_p\in T_pS\}\subset\g$$ 
  coincides with
  $N_{\g}(\g_p)=N_{\g}(\g_a)$,
\end{enumerate}
for all $p\in S$.
\end{thr}
\begin{proof}
  Since $\g_p$ with $p\in S$ is independent of~$p$, it will suffice to prove
  the theorem only for an arbitrary point $p\in S$, say $p=a$. Furthermore,
  since all assertions of the theorem have local character, we can restrict
  our consideration to the neighborhood $U\subset M$ where $S$~is given by
  equations~(\ref{2}). 

  1. It suffices to verify that all the functions $f\in \F$ are first integrals
  of~$E$ in~$U$. If $Y\in\g_a$ and $Z\in \D(E)$, then by~(\ref{1}) we
  have 
\begin{multline*}
  [Z,Y-f_1X_1-\dots-f_nX_n]=[Z,Y]-f_1[Z,X_1]-\dots-f_n[Z,X_n]\\ 
  -Z(f_1)X_1-\dots-Z(f_n)X_n\in\D(E).
\end{multline*}
Now since $Y,~X_1,\dots,~X_n\in\sym(E)$, we see that
$[Z,Y],~[Z,X_1],\dotsc$, $[Z,X_n]\in\D(E)$, and hence
$$ Z(f_1)X_1+\dots+Z(f_n)X_n\in\D(E).$$
But this is possible only if $Z(f_1)=\dots=Z(f_n)=0$, so that $f_1,\dots,f_n$
are indeed first integrals of~$E$.

2. Let $Z\in\g$. Since in a neighborhood of~$a$, the set $S$~is given
by~(\ref{2}), we see that $Z_a\in T_aS$ if and only if 
$d_af(Z_a)=Z(f)(a)=0$ for all $f\in\F$.

Using~(\ref{1}) with $Y\in\g_a$, we get
\begin{multline*}
  [Z,Y]=[Z, f_1X_1+\dots+f_nX_n]=Z(f_1)X_1 +\dots+Z(f_n)X_n\\ 
  f_1[Z,X_1]+\dots+f_n[Z,X_n] \pmod{\D(E)}.
\end{multline*}
This last equality, considered at the point~$a$, gives
$$ [Z,Y]_a=Z(f_1)(a)\,(X_1)_a +\dots+Z(f_n)(a)\,(X_n)_a \pmod{E_a}.$$
It follows that the condition $[Z,Y]\in\g_a$ is equivalent to
$$Z(f_1)(a)=\dots=Z(f_n)(a)=0.$$ 
Thus $Z_a$ lies in $T_aS$ if and only if $[Z,Y]\in\g_a$ for all $Y\in\g_a$,
that is if $Z\in N_{\g}(\g_a)$. 
\end{proof}

\begin{cor}
  If the subalgebra $\g_a$ coincides with its own normalizer, then $S$ is
  a maximal integral manifold of the distribution~$E$.
\end{cor}
\begin{proof}
  Indeed, on the one hand, $T_pS\supset E_p$ for all  $p\in S$, but on the other
  hand, $T_pS\cap\g_p=E_p\cap\g_p$ for all $p\in S$, which is possible
  only if $T_pS=E_p$, so that $S$ is a maximal integral manifold of~$E$. 
\end{proof}

Consider the restriction~$\tilde E$ of the distribution~$E$ to~$S$ and also
the set  
$\tilde {\g}=N(\g_a)$, restricted to $S$. Then $\tilde {\g}$ is clearly a
subalgebra of $\D(S)$ (which need not be isomorphic with $N(\g_a)$), and as
before, $\tilde {\g}\subset\sym(\tilde E)$. 

\begin{prop}\label{pn}
  The symmetry algebra $\tilde {\g}$ of~$\tilde E$ is simply transitive, 
  and the ideal $\tilde{\aa}=\tilde{\g}\cap\D(\tilde E)$ coincides with~$\g_a$. 
\end{prop}
\begin{proof}
  The transitivity of $\tilde {\g}$ follows from
\begin{multline}
  T_pS=T_pS\cap(E_p+\g(p))=E_p+T_pS\cap\g(p)=E_p+N(\g_p)(p)\\ =
  E_p+N(\g_a)(p)= E_p+\tilde {\g}(p) \quad \text{for all }p\in S.
\end{multline}

Further, it is clear that $\tilde {\g}_p=\g_p=\g_a$ for all $p\in S$, and
therefore $\tilde{\aa}=\cap_{p\in S} \tilde {\g}_p=\g_a$, so that $\tilde {\g}$
is a simply transitive symmetry algebra.  
\end{proof}

Thus, {\em the problem of integration of a distribution with the help of
  symmetries can be divided into the following two parts\/}:
\begin{enumerate}
\item the construction of manifolds $S_a$;
\item the integration of the distributions $\tilde E$ on each of the
  manifolds~$S_a$ by means of the simply transitive symmetry algebras~$\tilde{\g}$.
  \end{enumerate} 

\subsection{Simply transitive symmetry algebras and $\g$-structures}

Let $E$ be a completely integrable distribution on~$M$, let $\h$ be a simply
transitive symmetry algebra of~$E$, let $\aa=\h\cap\D(E)$ be the ideal in~$\h$
consisting of
characteristic symmetries, and finally let $\g=\h/\aa$.

We define a $\g$-valued 1-form $\omega$ on~$M$ by requiring that
\begin{enumerate}
\item[(i)] $\omega(Y)=0$ for all $Y\in\D(E)$;
\item[(ii)] $\omega(X)=X+\aa$ for all $X\in\h$.
\end{enumerate}
It is easy to verify that $\omega$~is well-defined. 

\begin{prop}\label{p1}
  The form $\omega$ has the following properties:
  \begin{enumerate}
  \item $d\omega(X_1,X_2)=-[\omega(X_1),\omega(X_2)]$ for all 
    $X_1,X_2\in\D(M)$;
  \item $\ker \omega_p=E_p,\ \im \omega_p=\g$ for all $p\in M$.
  \end{enumerate}
\end{prop}
\begin{proof}

  1. Since a vector field on~$M$ may be (uniquely) written in the form
  $$ f_1X_1+\dots+f_nX_n+Y, \qquad f_1,\dots,f_n\in C^\infty(M),\ Y\in\D(E)$$ 
and since both sides of the desired equality are 
$C^\infty(M)$-bilinear, we need to verify this equality only when
\begin{enumerate}
\item[(i)] $X_1,X_2\in \D(E)$;
\item[(ii)] $X_1\in\D(E),\ X_2\in\h$;
\item[(iii)] $X_1,X_2\in\h$.
\end{enumerate}
In case (i) both sides of our equality vanish identically. Note that 
$\omega(X)=\const$ for all $X\in\h$. We have
$$d\omega(X_1,X_2)=-\omega([X_1,X_2])+X_1\omega(X_2)-X_2\omega(X_1)$$ for all
$X_1,X_2\in\D(M)$. It follows that in case~(ii)  both sides of the
equality  are also zero, because $[X_1,X_2]\in\D(E),\ \omega(X_1)=0$, and
$\omega(X_2)=\const$. 

In case (iii) we have $[X_1,X_2]\in\h$, so that 
\begin{multline*}
  d\omega(X_1,X_2)=-\omega([X_1,X_2])=-[X_1,X_2]+\aa\\ 
  =-[X_1+\aa,X_2+\aa]=-[\omega(X_1),\omega(X_2)],
\end{multline*} 
as was to be proved.

\nopagebreak
2. The second statement follows immediately from the definition of~$\omega$. 
\end{proof}

Recall that a {\em $\g$-structure on the manifold~$M$\/} is a
$\g$-valued 1-form~$\omega$ satisfying the first condition of Proposition~
\ref{p1}. We shall say that a $\g$-structure $\omega$ is {\em nondegenerate\/}
if $\im \omega_p=\g$ for all $p\in M$. We thus see that any simply transitive
symmetry algebra of~$E$ defines a nondegenerate $\g$-structure on~$M$.

Conversely, a nondegenerate $\g$-structure $\omega$ determines a distribution~
$E$ on~$M$ by assigning to each point $p\in M$ the subspace $E_p=\ker
\omega_p$. Moreover, for any $\overline X\in\g$, the equality
$\omega(X)=\overline X$ determines a unique (up to $\D(E)$) vector field~$X$ 
on~$M$. 

\begin{prop}\ 

  1. The distribution $E$ given by a $\g$-structure $\omega$ is completely
integrable. 

  2. For any $\overline X\in\g$, the vector field $X\in\D(E)$ is a symmetry
of~$E$. 

  3. The set $\h=\{X\in\D(E)\mid \omega(X)=\const\}$ forms a simply transitive
symmetry algebra of~$E$ with the following properties:
\begin{itemize}
\item[(i)] $\h\supset\D(E)$ and $\h/\D(E)\cong\g$;
\item[(ii)] any simply transitive symmetry algebra of~$E$ which determines the
 $\g$-structure $\omega$ is contained in~$\h$.
\end{itemize}  
\end{prop}
\begin{proof}\ 

  1. Indeed, given $Y_1,Y_2\in\D(E)$, we have
\begin{multline*}
  \omega([Y_1,Y_2])=-d\omega([Y_1,Y_2])+Y_1\omega(Y_2)-Y_2\omega(Y_1)\\ 
  =[\omega(Y_1),\omega(Y_2)]+Y_1\omega(Y_2)-Y_2\omega(Y_1)=0,
\end{multline*}
since $\omega(Y_1)=\omega(Y_2)=0$.  Therefore $[Y_1,Y_2]\in\D(E)$, and the
distribution~$E$ is completely integrable. 

2. If $\overline{X}\in\g$ and $Y\in\D(E)$, we have
$$ \omega([X,Y])=[\omega(X),\omega(Y)]+X\omega(Y)-Y\omega(X)=0,$$ because
$\omega(Y)=0$ and $\omega(X)=\overline{X}=\const$.

3. The inclusion $\h\supset\D(E)$ is obvious. It is immediate from the
definitions that the mapping $\h\to\g$ such that $X\mapsto\omega(X)$ is a
surjective homomorphism of Lie algebras, and its kernel coincides with
$\D(E)$. Hence $\h/\D(E)\cong\g$. If a simply transitive symmetry algebra~$\h'$
determines the same $\g$-structure~$\omega$, then $\omega(X)=\const$ for all
$X\in\h'$, so that $\h'\subset\h$. 
\end{proof}

We have thus proved that {\em completely integrable distributions on~$M$
 with simply
transitive symmetry algebras are in one-to-one correspondence with
nondegenerate $\g$-structures on~$M$.}

\subsection{Integration of $\g$-structures}

\subsubsection{Integrals}
Let $\omega$ be a $\g$-structure on~$M$, not necessarily nondegenerate, and let
$G$~be a Lie group whose Lie algebra is isomorphic with~$\g$. We identify
$\g$ with $T_eG$ and also, by means of the \textbf{right} 
translations, with all tangent
spaces~$T_gG$. 

\begin{defi}
  A mapping $f\colon M\to G$ is called an {\em integral of the $\g$-structure
    $\omega$\/} if the differential $d_pf\colon T_pM\to T_{f(p)}G\equiv\g$ of
~$f$ coincides with $\omega_p$ for all $p\in M$.
\end{defi}

A $\g$-structure $\omega$ is said to be integrable if there exists an integral 
$f\colon M\to G$ of~$\g$. 

\begin{thr}\label{t2}
  Any $\g$-structure $\omega$ is locally integrable. Moreover, given
  $a\in M$ and $g_0\in G$, there is a neighborhood~$U$ of~$a$ such that there
exists a unique integral $f$ of $\omega|_U$ satisfying the ``initial''
condition $f(a)=g_0$.
\end{thr}
\begin{proof}
  Consider the distribution~$H$ on $M\times G$ defined by 
  $$H_{(p,g)}=\{X_p+\omega_p(X_p)\mid X_p\in T_pM\},\quad p\in M,\,g\in G.$$
  It is easy to check that $H$ is completely integrable and that the
dimension of~$H$ is equal to $\dim M$. If $\pi$ is the natural projection of
the direct product $M\times G$ onto~$M$, then, as easily follows from the
definition of~$H$, the mapping $d_{(p,g)}\pi$ determines a homomorphism of  
$H_{(p,g)}$ onto $T_pM$ for all $(p,g)\in M\times G$.

  It follows from the Frobenius theorem that there is a unique integral
manifold~$L$ of the distribution~$H$ that passes through $(a,g_0)$, and
$\pi|_L$ is a local diffeomorphism of $L$ and~$M$ at the point $(a,g_0)$.
  Therefore, in some neighborhood~$U$ of~$a$, there is a unique mapping
  $f\colon U\to M$ whose graph coincides with $L\cap(U\times G)$. It now
follows that $f$ is an integral of the $\g$-structure $\omega|_U$. 
\end{proof}

\begin{ex}
  Consider the trivial case that $G=\R$ and $\g$ is the one-dimensional
commutative Lie algebra. The corresponding $\g$-structure $\omega$ is then an
ordinary closed 1-form on~$M$, and its integrals are just integrals of a closed
form, that is the functions $f\colon M\to G\equiv\R$ such that
  $df=\omega$. Thus the integration of $\g$-structures may be viewed as a
generalization of the integration of closed $1$-forms.
\end{ex}

Let $H$ be the distribution on $M\times G$ defined by the
$\g$-structure~$\omega$. 
Then $H$~may be regarded as a flat connection on the trivial
principal $G$-bundle $\pi\colon M\times G\to G$. Since the connection~$H$ is
flat, there is a natural homomorphism $\phi\colon\pi_1(M)\to\Phi$ of the
fundamental group of~$M$ into the holonomy group of the connection~$H$, 
and $\omega$ is
globally integrable if and only if the group~$\Phi$ is trivial. In the general
case, $\phi$ is surjective, and any integrable manifold of the distribution~$H$
is a covering  of~$M$ with respect to the  projection~$\pi$ with
fiber~$\Phi$. It follows in particular that any $\g$-structure on a simply
connected manifold is globally integrable. 

Assume that the $\g$-structure $\omega$ is globally integrable, and let $f$~be
an integral of~$\omega$. Then the mapping $g.f\colon M\to G,\ p\mapsto f(p)g$
is obviously an integral of~$\omega$ too. Thus we have a right action of~$G$ on
the set of all integrals of~$\omega$. If, in addition, $M$ is connected, then 
by Theorem~\ref{t2} this action is transitive. In the trivial case that 
$G=\R$, this statement brings us to the well-known result that the integral of
an exact 1-form is unique up to the addition of an arbitrary constant. 

\subsubsection{Integrals of $\g$-structures and distributions}

Let $\omega$ be a nondegenerate $\g$-structure, and let $E$ be the completely
 integrable distribution defined by~$\omega$. Then every (local) integral of~ 
$\omega$ is a submersion, and therefore $f^{-1}(g)$, for any $g\in
G$, can be given a submanifold structure. 

\begin{prop}
  If $f$ is an integral of the $\g$-structure $\omega$, then for any $g\in G$,
  the connected components of the submanifold $f^{-1}(g)$ will be maximal 
  integral manifolds of the distribution~$E$.
\end{prop}
\begin{proof}
  If $L$ is a connected component of $f^{-1}(g)$, then 
$$T_pL=\ker d_pf=\ker \omega_p=E_p$$
for all $p\in L$, so that $L$ is an integral manifold of~$E$.
Assume that $L'$ is an integral manifold of~$E$ containing~$L$. Since $L'$ is 
connected and $d_pf(T_pL')=d_pf(E_p)=\{0\}$ for all $p\in L'$, the mapping 
$f$ is constant on~$L'$. Thus $L'\subset L$, and therefore $L=L'$.   
\end{proof}

Thus the integration of the distribution~$E$ reduces essentially to the
integration of the corresponding $\g$-structure.

\subsubsection{Reduction}

Let $G_1$ be a normal Lie subgroup of~$G$, and let $\g_1$ be the corresponding
ideal in~$\g$. Consider the quotient group $G_2=G/G_1$, whose Lie algebra~$\g_2$
is $\g/\g_1$. Suppose that $\pi\colon G\to G_2$ is the canonical surjection
 and $d\pi\colon\g\to\g_2$ the corresponding surjection of Lie algebras. The
manifold~$M$ can be supplied with a natural $\g_2$-structure $\omega_2$:
$$(\omega_2)_p=d\pi\circ\omega_p\quad\text{ for all }p\in M.$$
Let $f$~be an integral of the $\g$-structure $\omega$. Then 
$\pi\circ f$ is obviously an integral of the $\g_2$-structure $\omega_2$. 

Now assume that  the $\g$-structure $\omega$ is nondegenerate; then the
$\g_2$-structure~$\omega_2$ is also nondegenerate. If $f_2$ is an integral
of~$\omega_2$ and $L=f_2^{-1}(g)$ is the inverse image of an arbitrary point
$g\in G_2$, consider the restriction $\omega_1=\omega|_L$. Then
$(\omega_1)_p(T_pL)\subset \ker d\pi=\g_1$ for all $p\in L$. Thus $\omega_1$
may be considered as a $\g_1$-structure on~$L$. It is easy to show that
$\omega_1$~is also nondegenerate. Integrating~$\omega_1$
for all submanifolds $f_2^{-1}(g),\ g\in G_2$, we obtain an integral of the
$\g$-structure~$\omega$.  

Thus the problem of integrating a nondegenerate $\g$-structure~$\omega$
can be divided into smaller parts:
\begin{enumerate}
\item the construction of an integral $f_2$ of the $\g_2$-structure $\omega_2$;
\item the integration of the  $\g_1$-structures
$\omega_1=\omega|_{f_2^{-1}(g)}$ for $g\in G_2$. 
\end{enumerate}

\begin{prop}\ \label{p3} 
  
1. The integration of any $\g$-structure with a solvable Lie algebra~$\g$
reduces to the integration of closed $1$-forms.

  2. The integration of any $\g$-structure reduces to the integration of 
  $\g$-structures with simple Lie algebras~$\g$ and to the integration of
  closed $1$-forms. 
\end{prop}
\begin{proof}\
 
1.  We may assume without loss of generality that the Lie group~$G$
is connected and simply connected. Then $G$ is diffeomorphic to $\R^k$ for some
 $k\in\mathbb N$, and moreover any normal virtual Lie subgroup of~$G$ is closed
and also simply connected. 

  Since the Lie algebra~$\g$ is solvable, there is a chain 
  $$ \g=\g_k\subset\g_{k-1}\subset\dots\subset\g_1\subset\g_0=\{0\}$$ 
of subalgebras of~$\g$ such that $\g_{i-1}$ is an ideal of codimension~$1$ in
$\g_i$ ($i=1,\dots, k$). If $G_i$ are the corresponding subgroups of~$G$, then
all quotient subgroups $G_i/G_{i-1}$ are isomorphic with~$\R$, and the problem
of finding an integral of the $\g$-structure~$\omega$ reduces to the
integration of $k$ differential 1-forms.

2. The second statement of the theorem follows from the Levi theorem about 
the decomposition of a Lie group into the semidirect product of a semisimple
Levi subgroup and the radical, and from the decomposition of a simply
connected semisimple Lie group into a direct product of simple Lie groups. 
\end{proof}

\subsubsection{Integration along paths}

In conclusion, we describe a procedure for finding an integral~$f$ for a given
$\g$-structure~$w$, a procedure that generalizes the process of integrating
1-forms along paths and coincides with it when $G=\R$. Recall that a
$\g$-structure $\omega$ may be regarded as a connection~$H$ on the principal
$G$-bundle $\pi\colon M\times G\to M$.

Let $(a,g_0)$ be a fixed point of the manifold $M\times G$; our task is to
find an integral~$f$ of~$\omega$ such that $f(a)=g_0$. 
Consider an arbitrary curve $\Gamma\colon[t_0,t_1]\to M$ with $\Gamma(t_0)=a$.
There exists a unique curve
$\Tilde\Gamma\colon[t_0,t_1]\to M\times G$ satisfying the following conditions:
\begin{enumerate}
\item $\pi\circ\Tilde\Gamma=\Gamma$;
\item $\Tilde\Gamma'(t)\in H_{\Tilde\Gamma(t)}$ for all $t\in(t_0,t_1)$.
\end{enumerate}
This curve is precisely the horizontal lift of the curve $\Gamma$ by means of
the connection~$H$. 

The curve $\Tilde\Gamma$ may be equivalently described as follows:
Consider the $\g$-structure~$\omega_1$ defined on an interval $[t_0,t_1]$ by
$$ \omega_1=\omega\circ d\Gamma=X(t)dt,$$
where
$$ X(t)=\omega_{\Gamma(t)}(\Gamma'(t)),\qquad t\in[t_0,t_1]$$
is a curve in~$\g$. 
Then the desired integral of~$\omega_1$ is a curve $g\colon[t_0,t_1]\to G$ in
$G$ satisfying the differential equation
\begin{equation}\label{3}
g'(t)=X(t),\ t\in[t_0,t_1],\qquad  g(t_0)=g_0.
\end{equation} 
It is not hard to show that the curve $\Tilde\Gamma$ on the manifold
$M\times G$ has the form $\Tilde\Gamma(t)=(\Gamma(t),g(t))$. 

Let $L$ be the maximal integral manifold of the distribution~$H$ on $M\times
G$ passing through $(a,t_0)$. The desired integral $f$ of the $\g$-structure 
$\omega$ is a mapping $M\to G$ whose graph coincides with~$L$.   
Since the curve~$\Tilde\Gamma$ is tangent to~$H$, it lies inside~$L$, and
therefore 
$$f(\Gamma(t))=g(t)\qquad\text{for all }t\in[t_0,t_1].$$

The method described above enables us to find the integral of the
$\g$-structure $\omega$ along any curve~$\Gamma$ on~$M$. 
When $G=\R$, this method is the same as the usual procedure for the integration
of 1-forms. The standard results of the theory of connections show that the
point $f(\Gamma(t_1))$ depends only on the homotopy class of the curve 
$\Gamma(t)$. Thus our method allows to determine the integral~$f$ uniquely in
any simply connected neighborhood of~$a$.

\subsection{Integration of distributions}

 The next theorem is a
consequence of Proposition~\ref{p3}.
\begin{thr}
  Let $\g$ be a transitive symmetry algebra of the distribution~$E$, and
  assume that there is a point $a\in M$ such that the Lie algebra
  $N(\g_a)/\g_a$ is solvable. Then $E$ can be integrated by quadrature.
\end{thr}

In the most general case the problem of integration of a completely integrable
distribution with a transitive symmetry algebra can be divided into the
following three parts:
\begin{enumerate}
\item[(1)] the reduction of the problem to the integration of a
  distribution with a simply transitive symmetry algebra isomorphic to 
  $N(\g_a)/\g_a$; in terms of coordinates, this is equivalent to the solution
  of simultaneous equations (which are, in general, transcendental);
\item[(2)] the integration of $\g$-structures with simple Lie algebras~$\g$
  from the decomposition of the Levi subalgebra of $N(\g_a)/\g_a$; in terms of
  coordinates, this is equivalent to the solution of finitely many ordinary
  differential equations of the form~(\ref{3});
\item[(3)] the integration of a $\g$-structure, where the Lie algebra~$\g$ is
  solvable and coincides with the radical of $N(\g_a)/\g_a$; this reduces to
  the integration of a finite number (equal to $\dim\g$) of $1$-forms, which,
  in terms of coordinates, is equivalent to usual integration.
\end{enumerate}

\subsection{Superposition principle}

Let $\omega$ be an arbitrary $\g$-structure on~$M$, where $\g$~is a
Lie algebra of a Lie group~$G$. Assume that $G$ effectively acts
on a certain manifold~$N$. Then there is a natural homomorphism 
$\alpha\colon\g\to\D(N)$ of Lie algebras. For each $q\in N$, consider the
mapping $\alpha_q\colon \g\to T_qN$ defined by $\alpha_q(X)=\alpha(X)_q$.

Let us fix a point $a\in M$ and try to find a (local) integral~$f$
of~$\omega$ satisfying the initial condition $f(a)=e$.

\begin{prop}\label{p6}\ 

1. Let $f$ denote the desired integral of the $\g$-structure~$\omega$, and let
$b$~be a point of~$N$. Then the mapping 
$F\colon M\to N,\ p\to f(p).b$ satisfies the following differential 
equation of the first order:  
\begin{equation}\label{eqmn}
d_pF=\alpha_{F(p)}\circ\omega_p, \quad p\in M.
\end{equation}

2. If $a\in M$ and $b\in N$ are any points, then equation~(\ref{eqmn}) with
initial condition $F(a)=b$ has a unique solution in some neighborhood of~$a$.
\end{prop}  
\begin{proof}\ 

 1. Indeed, we can write $F$ as the composition of the mappings
  $$ M \stackrel{f}{\rightarrow} G\stackrel{\pi}{\rightarrow} G\times
  N\stackrel{\sigma}{\rightarrow} N,$$ 
where $\pi(g)=(g,b)$ and $\sigma$ is the action of~$G$ on~$N$.  The
composition of the differentials of these mappings is precisely
$\alpha_{F(p)}\circ\omega_p$. 

2. Here it suffices to note that equation~(\ref{eqmn}) determines a completely
integrable distribution on $M\times N$ whose maximal integral manifolds are 
precisely the graphs of the solutions of~(\ref{eqmn}).    
\end{proof}

Equation~(\ref{eqmn}) satisfies the following {\em superposition principle\/}: 

\begin{thr}
There exist a number $k\in \mathbb N$, a smooth function 
$$\Phi\colon\underbrace{N\times\dots\times N}_{k+1\text{ \upshape times}}
\to N,$$
(independent of~$\omega$), and $k$~particular solutions,
$F_1,\dots,F_k$, of equation~(\ref{eqmn}) such that  any solution 
of~(\ref{eqmn}) satisfying the initial condition $F(a)=b$ has the form
$$F(p)=\Phi(F_1(p),\dots,F_k(p),b).$$ 
\end{thr}
\begin{proof}
  Let $b_1,\dots,b_k$ be $k$ points in~$N$ such that 
  $\cap_{i=1}^k G_{b_k}=\{e\}$. The existence of such $k$~points follows from
  the effectiveness of the action of~$G$ on~$N$. Let now $F_1,\dots,F_k$ be
  solutions of~(\ref{eqmn}) satisfying the initial conditions $F_i(a)=b_i,\
  i=1,\dots,k$. Consider the orbit $\mathcal O$ of the point $(b_1,\dots,b_k)$
  under the componentwise action of~$G$ on the manifold
  $\underbrace{N\times\dots\times N}_{k\text{ \upshape times}}$.
  The action of $G$ on~$\mathcal O$ is simply transitive, so that for any point
  $(q_1,\dots,q_k)\in \mathcal O$, there is a unique element
  $g(q_1,\dots,g_k)\in G$ such that
$$g(q_1,\dots,q_k).(b_1,\dots,b_k)=(q_1,\dots,q_k).$$
Let  
$$\Phi(q_1,\dots,q_k,b)=g(q_1,\dots,q_k).b$$ 
and show that the function~$\Phi$ indeed satisfies the required property. 
If $f$ is an integral of the $\g$-structure $\omega$ and if $f(a)=e$, then, by
Proposition~\ref{p6}, the functions~$F_i$ have the form $p\mapsto f(p).b_i$ for
all $i=1,\dots,k$. It follows that
$$ f(p)=g(F_1(p),\dots,F_k(p)), \quad\text{for all }p\in M.$$
Thus the mapping $F\colon p\mapsto f(p).b$ is a solution of 
equation~(\ref{eqmn}) with initial condition $F(a)=b$ and can be written as
follows:
$$ F(p)=g(F_1(p),\dots,F_k(p)).b=\Phi(F_1(p),\dots,F_k(p),b).$$ 
\end{proof}

\begin{ex}
Suppose that $M=\R$. Then a $\g$-structure~$\omega$ on~$M$ may be identified
with the curve $X(t),\ t\in \R$, in the Lie algebra~$\g$ such that
$\omega=X(t)\,dt$. Equation~(\ref{eqmn}) will then have the form
\begin{equation}\label{eqs1}
F'(t)=\alpha(X(t))_{F(t)}
\end{equation}
and will be a non-autonomous ordinary differential equation of the first order
on the manifold~$N$. By the above theorem, in order to construct the general
solution of this equation, we must find only a finite number of its particular
solutions. 

  If, for example, $N=\R^n$ and $G$ is the group of all linear transformations,
  then, in the standard coordinate system of~$N$, equation~(\ref{eqs1}) is
  written as a system of linear first-order differential equations:
\begin{equation}\label{eqs2}
F'(t)=A(t)F(t),
\end{equation}
with $F(t)\in \R^n$ and $A(t)\in\Mat_n(\R)$, while the superposition principle
becomes the well-known result that the general solution of such an equation is
a linear combination of $n$~particular solutions with linearly independent
initial conditions. 
\end{ex}

\end{document}